\documentclass[11pt,a4paper,reqno]{amsart}

\usepackage{fullpage}
\usepackage{amsmath,amssymb,amsthm,mathtools,mathrsfs}
\usepackage{aliascnt}
\usepackage[hidelinks]{hyperref}
\usepackage[nameinlink,noabbrev]{cleveref}

\numberwithin{equation}{section}
\crefname{equation}{}{}

\newtheorem{theorem}{Theorem}[section]

\newaliascnt{proposition}{theorem}
\newtheorem{proposition}[proposition]{Proposition}
\aliascntresetthe{proposition}

\newaliascnt{lemma}{theorem}
\newtheorem{lemma}[lemma]{Lemma}
\aliascntresetthe{lemma}

\newaliascnt{corollary}{theorem}
\newtheorem{corollary}[corollary]{Corollary}
\aliascntresetthe{corollary}

\newaliascnt{definition}{theorem}
\newtheorem{definition}[definition]{Definition}
\aliascntresetthe{definition}

\theoremstyle{remark}

\newaliascnt{remark}{theorem}
\newtheorem{remark}[remark]{Remark}
\aliascntresetthe{remark}

\crefname{theorem}{Theorem}{Theorems}
\crefname{proposition}{Proposition}{Propositions}
\crefname{lemma}{Lemma}{Lemmas}
\crefname{corollary}{Corollary}{Corollaries}
\crefname{definition}{Definition}{Definitions}
\crefname{remark}{Remark}{Remarks}

\crefname{equation}{}{}

\newcommand{\R}{\mathbb R}
\newcommand{\ii}{\mathrm i}
\newcommand{\ee}{\mathrm e}
\newcommand{\dd}{\,\mathrm d}
\newcommand{\norm}[1]{\left\lVert #1\right\rVert}
\newcommand{\snorm}[1]{\lVert #1\rVert}
\newcommand{\abs}[1]{\left\lvert #1\right\rvert}

\newcommand{\cC}{\mathcal C}
\newcommand{\cB}{\mathcal B}

\title{A Trichotomy for Modified Scattering\\
	Across the Yukawa--Coulomb Transition}

\author{Yonggeun Cho}
\address[Y. Cho]{Department of Mathematics and Institute of Pure and Applied Mathematics, Jeonbuk National University, Jeonju 54896, Republic of Korea}
\email{\href{mailto:changocho@jbnu.ac.kr}{changocho@jbnu.ac.kr}}

\author{Jinyeop Lee}
\address[J. Lee]{Department of Applied Mathematics, Kyung Hee University, 1732 Deogyeong-daero, Giheung-gu, Yongin-si, Gyeonggi-do, South Korea}
\email{\href{mailto:jinyeop.lee@khu.ac.kr}{jinyeop.lee@khu.ac.kr}}

\date{}

\begin{document}
\begin{abstract}
We study the long-time asymptotics of the three-dimensional Hartree equation with Yukawa potential
\[
V_\mu(x)=\frac{\ee^{-\mu\abs{x}}}{\abs{x}}, \qquad 0\leq\mu\leq1.
\]
The Coulomb case corresponds to $\mu=0$, while $\mu>0$ introduces the screening length $\mu^{-1}$. In the limit $\mu\to0$ and $t\to\infty$, the asymptotic behavior depends on the comparison between the observation scale $t$ and the screening length $\mu^{-1}$, equivalently on the parameter $\mu t$. This leads to three distinct asymptotic regimes, according as $\mu t\to0$, $\mu t\to L\in(0,\infty)$, or $\mu t\to\infty$, with different modified scattering phases in each case.
\end{abstract}

\maketitle

\section{Introduction}
\label{sec:introduction-history}

We consider the Cauchy problem for the three-dimensional Hartree equation with Yukawa--Coulomb interaction
\begin{equation}\label{eq:hartree-yukawa-coulomb}
\begin{cases}
\displaystyle \ii\partial_t u_\mu = -\frac12\Delta u_\mu + \kappa \bigl(V_\mu*\abs{u_\mu}^2\bigr)u_\mu,\\[0.4em]
u_\mu(0)=u_{\mathrm{in}},
\end{cases}
\end{equation}
where $\kappa\in\{+1,-1\}$. We take $0\leq\mu\leq1$ in the equation and define the interaction kernel for every $\lambda\geq0$ by
\begin{equation}\label{eq:Yukawa-potential-definition}
V_\lambda(x)=\frac{\ee^{-\lambda\abs{x}}}{\abs{x}},\qquad \lambda\geq0.
\end{equation}
The endpoint $\mu=0$ is understood as the Coulomb potential $V_0(x)=\abs{x}^{-1}$.

For fixed $\mu>0$, the Yukawa potential is exponentially screened at distances larger than the screening length $\mu^{-1}$. For $\mu=0$, the Coulomb potential is unscreened and produces the logarithmic long-range phase.

In the Thomas--Fermi approximation for a metal or electron gas, the Coulomb Fourier multiplier $4\pi/\abs{\xi}^2$ is replaced by $4\pi/(\abs{\xi}^2+\mu^2)$, where $\mu$ is the inverse screening length. See Ashcroft--Mermin~\cite[Chap.~17, pp.~341--344]{AshcroftMermin1976}. This approximation can also be viewed as the static long-wavelength limit of the Lindhard dielectric response~\cite[pp.~8--36]{Lindhard1954}.

The scattering behavior of \eqref{eq:hartree-yukawa-coulomb} changes discontinuously at the endpoint $\mu=0$. For fixed $\mu>0$, the Yukawa potential is exponentially decaying and hence belongs to the short-range side of Hartree scattering theory. See, for instance,~\cite{ChoHwangOzawa2016}. Masaki~\cite{Masaki2019Hartree} established scattering criteria in scaling-critical weighted spaces for mass-subcritical Hartree equations with homogeneous short-range interactions. By contrast, the Coulomb Hartree equation is long-range, and the usual linear scattering has to be replaced by modified scattering with a logarithmic phase. This phenomenon was developed in the classical works of Ginibre--Ozawa~\cite{GinibreOzawa1993}, Hayashi--Naumkin~\cite{HayashiNaumkin1998AJM}, and Hayashi--Naumkin--Ozawa~\cite{HayashiNaumkinOzawa1998SIMA}. Related constructions of modified wave operators for long-range Hartree equations were given by Ginibre--Velo~\cite{GinibreVelo2000I,GinibreVelo2000II} and Nakanishi~\cite{Nakanishi2002}. A Fourier-space and stationary-phase approach, close in spirit to wave-packet arguments, is due to Kato--Pusateri~\cite{KatoPusateri2011}. More recent approaches to modified scattering for Hartree-type equations include low-regularity and wave-packet methods~\cite{VanHoose2024Hartree,VanHoose2025Nonlocal} and infinite-rank Coulomb Hartree dynamics~\cite{NguyenYou2025}.

The point of the present paper is different from both the fixed Coulomb problem and the fixed short-range problem. We study the singular two-parameter limit
\[
\mu\to0, \qquad t\to\infty.
\]
At first sight, one might expect that the limit $\mu\to0$ simply recovers the usual Coulomb modified scattering phase. This is not the case. The long-time asymptotics are governed by the comparison between the dispersive scale $t$ and the screening length $\mu^{-1}$, or equivalently by the single parameter $\mu t$. As a consequence, the Coulomb limit and the long-time limit do not commute. Along sequences $\mu_n\to0$ and $t_n\to\infty$, the effective transition is controlled by whether the observation scale $t_n$ is smaller than, comparable to, or larger than the screening length $\mu_n^{-1}$.

\subsection{Notation and phase functionals}
\label{subsec:notation-function-spaces}
\label{subsec:phase-functionals}

We define the weighted Sobolev space
\begin{equation}\label{eq:Sigma-four-definition}
\Sigma_4 := \left\{ u\in H^4(\R^3): \abs{x}^4u\in L^2(\R^3) \right\}, \qquad \snorm{u}_{\Sigma_4} := \norm{u}_{H^4} + \snorm{\abs{x}^4u}_{L^2}.
\end{equation}
We use the notation
\begin{equation}\label{eq:oLp-notation}
F_n = G_n + o_{L_x^p}(1)
\end{equation}
to mean that
\[
\norm{F_n-G_n}_{L_x^p} \to 0 \qquad \text{as }n\to\infty.
\]

We define the \emph{Coulomb phase functional} by
\begin{equation}\label{eq:Coulomb-phase-functional}
\cC[G](v) := \kappa \int_{\R^3} \frac{\abs{G(w)}^2}{\abs{v-w}} \,\dd w.
\end{equation}
For $L\in[0,\infty)$, define the \emph{transition kernel}
\begin{equation}\label{eq:transition-kernel-bL}
b_L(z) := \frac1{\abs{z}} \int_0^L \frac{\ee^{-r\abs{z}}-1}{r} \,\dd r, \qquad z\neq0.
\end{equation}
The corresponding \emph{finite transition phase} is
\begin{equation}\label{eq:finite-transition-phase}
\cB_L[G](v) := \kappa \int_{\R^3} b_L(v-w)\abs{G(w)}^2 \,\dd w.
\end{equation}
The \emph{limiting transition kernel} is
\begin{equation}\label{eq:limiting-transition-kernel}
b_\infty(z) := \frac{-\gamma-\log\abs{z}}{\abs{z}}, \qquad z\neq0,
\end{equation}
where $\gamma$ is the Euler constant. The limiting Yukawa correction is
\begin{equation}\label{eq:limiting-Yukawa-correction}
\cB_\infty[G](v) := \kappa \int_{\R^3} b_\infty(v-w)\abs{G(w)}^2 \,\dd w.
\end{equation}

The functionals $\cB_L$ and $\cB_\infty$ arise from the decomposition of the frozen nonlinear phase. For a fixed profile $G$, define the \emph{frozen phase} by
\begin{equation}\label{eq:frozen-phase-definition}
\Phi_\mu[G](t,v) := \kappa \int_1^t \frac1s \left(V_{\mu s}*\abs{G}^2\right)(v) \,\dd s.
\end{equation}
Then the basic identity is
\begin{equation}\label{eq:frozen-phase-decomposition-preview}
\Phi_\mu[G](t) = \cC[G]\log t + \cB_{\mu t}[G] - \cB_\mu[G].
\end{equation}
This identity is the source of the three regimes in the main theorem.

\subsection{Main result}
\label{sec:main-result}
\label{subsec:main-result}

\begin{theorem}
\label{thm:main-trichotomy}
There exists $\varepsilon_0>0$ such that the following holds. Let $0<\varepsilon\leq\varepsilon_0$, and assume that
\begin{equation}\label{eq:main-theorem-smallness}
\norm{u_{\mathrm{in}}}_{\Sigma_4} \leq \varepsilon.
\end{equation}
For each $0\leq\mu\leq1$, let $u_\mu$ be the solution to \eqref{eq:hartree-yukawa-coulomb} with initial datum $u_{\mathrm{in}}$. Then there exists a profile
\begin{equation}\label{eq:main-limiting-profile-definition}
W_0 \in L^2 \cap L^\infty
\end{equation}
such that the following trichotomy holds. Let $\mu_n\to0$ and $t_n\to\infty$.

If $\mu_nt_n\to0$, then
\begin{equation}\label{eq:main-theorem-regime-small-mut}
u_{\mu_n}(t_n,x) = (\ii t_n)^{-3/2} \ee^{\frac{\ii\abs{x}^2}{2t_n}} \ee^{-\ii\cC[W_0](x/t_n)\log t_n} W_0(x/t_n) + o_{L_x^2}(1).
\end{equation}

If $\mu_nt_n\to L\in(0,\infty)$, then
\begin{equation}\label{eq:main-theorem-regime-finite-mut}
u_{\mu_n}(t_n,x) = (\ii t_n)^{-3/2} \ee^{\frac{\ii\abs{x}^2}{2t_n}} \ee^{-\ii\cC[W_0](x/t_n)\log t_n} \ee^{-\ii\cB_L[W_0](x/t_n)} W_0(x/t_n) + o_{L_x^2}(1).
\end{equation}

If $\mu_nt_n\to\infty$, then
\begin{equation}\label{eq:main-theorem-regime-large-mut}
u_{\mu_n}(t_n,x) = (\ii t_n)^{-3/2} \ee^{\frac{\ii\abs{x}^2}{2t_n}} \ee^{-\ii\cC[W_0](x/t_n)\log(1/\mu_n)} \ee^{-\ii\cB_\infty[W_0](x/t_n)} W_0(x/t_n) + o_{L_x^2}(1).
\end{equation}
\end{theorem}

\begin{remark}
\label{rem:main-profile-W}
The profile $W_0$ in \Cref{thm:main-trichotomy} is the asymptotic one appearing in the modified scattering theory (for instance see \cite{HayashiNaumkin1998AJM}).  We will use $W_\mu$ for the frozen-profile normalized asymptotic profile constructed later in \Cref{sec:exact-renormalization-asymptotic-profiles}. The construction proceeds through the exact renormalized profile $Z_\mu$, followed by a finite phase correction.
\end{remark}

\begin{remark}
\label{rem:main-noncommutation}
The first regime shows that the usual Coulomb modified scattering phase is recovered only when $\mu_nt_n\to0$. The third regime shows that, if $\mu_nt_n\to\infty$, the logarithmic phase is no longer a function of $t_n$ but instead saturates at $\log(1/\mu_n)$. Thus the limits $\mu\to0$ and $t\to\infty$ do not commute.
\end{remark}

\begin{remark}
\label{rem:main-interpretation-trichotomy}
The length scale $\mu^{-1}$ is the screening length of the Yukawa interaction. The parameter $\mu t$ compares the observation scale $t$ with this screening length. The regime $\mu t\to0$ corresponds to observing the solution before the screening length is reached. The regime $\mu t\to L\in(0,\infty)$ corresponds to the transition scale. The regime $\mu t\to\infty$ corresponds to observing the solution beyond the screening length.
\end{remark}

\begin{remark}
\label{rem:VH}
The smallness and high regularity assumption $\norm{u_{\mathrm{in}}}_{\Sigma_4}\leq\varepsilon$ is used in this paper in order to apply the Hayashi--Naumkin weighted-profile method uniformly in $0\leq\mu\leq1$. It would be interesting to weaken this smallness and regularity requirement by combining the present argument with the wave packet methods of \cite{VanHoose2024Hartree,VanHoose2025Nonlocal}. We do not pursue this direction here.
\end{remark}

\subsection{Strategy of the proof}
\label{subsec:strategy-proof}

We briefly describe the structure of the proof of \Cref{thm:main-trichotomy}. The argument has three main parts. First, we obtain estimates which are uniform in the screening parameter $0\leq\mu\leq1$. Second, we identify and remove the exact long-range phase in self-similar variables. Finally, we pass from the exact renormalized dynamics to the frozen-profile asymptotics and analyze the resulting phase in the three regimes $\mu t\to0$, $\mu t\to L\in(0,\infty)$, and $\mu t\to\infty$.

The starting point is the collection of uniform estimates in \Cref{sec:preliminary-estimates}. These include the energy-level $H^1$ bound, the estimates
\[
\snorm{u_\mu(t)}_{H_x^4} + \snorm{\abs{x}^4 f_\mu(t)}_{L_x^2} \leq C\varepsilon t^{C\varepsilon^2}, \qquad f_\mu(t)=\ee^{-\frac{\ii t}{2}\Delta}u_\mu(t),
\]
and the sharp decay bound
\[
\norm{u_\mu(t)}_{L_x^\infty} \leq C\varepsilon t^{-3/2}.
\]
The Hayashi--Naumkin argument is adapted uniformly in $\mu$, using the elementary domination of the Yukawa kernel by the Coulomb kernel. We also provide the uniform potential estimates which will be used throughout the proof.

We then pass, in \Cref{sec:rescaled-amplitude-frozen-phase}, to the self-similar amplitude
\[
a_\mu(t,v) := (\ii t)^{3/2} \ee^{-\frac{\ii t}{2}\abs{v}^2} u_\mu(t,tv).
\]
This amplitude satisfies the equation
\[
\partial_t a_\mu = \frac{\ii}{2t^2}\Delta_v a_\mu - \frac{\ii\kappa}{t} \left(V_{\mu t}*\abs{a_\mu}^2\right)a_\mu.
\]
The term $t^{-2}\Delta_v a_\mu$ is integrable in time by the uniform weighted-profile estimate. Hence the only non-integrable contribution is the explicit long-range phase generated by
\[
\frac1t \left(V_{\mu t}*\abs{a_\mu}^2\right).
\]

At this stage we also prove the frozen phase identity
\[
\Phi_\mu[G](t) = \mathcal C[G]\log t + \mathcal B_{\mu t}[G] - \mathcal B_\mu[G],
\]
which is the algebraic source of the trichotomy.

The next step, carried out in \Cref{sec:exact-renormalization-asymptotic-profiles}, is to remove the exact long-range phase. We define
\[
\Theta_\mu(t,v) := \kappa \int_1^t \frac1s \left(V_{\mu s}*\abs{a_\mu(s)}^2\right)(v) \,\dd s
\]
and
\[
b_\mu(t,v) := \ee^{\ii\Theta_\mu(t,v)}a_\mu(t,v).
\]
The cancellation gives
\[
\partial_t b_\mu(t) = \ee^{\ii\Theta_\mu(t)} \frac{\ii}{2t^2}\Delta_v a_\mu(t),
\]
which is integrable in time. Therefore $b_\mu(t)$ converges to a limiting profile $Z_\mu$, with a quantitative tail estimate. We then absorb the stationary Yukawa correction $\Gamma_\mu^Y$ into the profile and introduce the frozen-profile normalized state
\[
W_\mu := \ee^{-\ii\Gamma_\mu^Y}Z_\mu.
\]

Having constructed $W_\mu$, we replace the exact evolving phase by the frozen phase generated by this limiting profile. This is done in \Cref{sec:frozen-profile-asymptotics} and yields the uniform asymptotic formula
\[
a_\mu(t) = \ee^{-\ii\Phi_\mu[W_\mu](t)}W_\mu + o_{L_v^2}(1).
\]
Returning to physical space gives the corresponding asymptotic formula for $u_\mu(t,x)$, still expressed in terms of $W_\mu$ and $\Phi_\mu[W_\mu]$.

It remains to understand the dependence of $W_\mu$ and of the phase functionals on $\mu$. The profile continuity is proved in \Cref{sec:quantitative-continuity-profiles}. The key estimate is
\[
\log(1/\mu) \norm{W_\mu-W_0}_{H_v^1} \longrightarrow0 \qquad \text{as }\mu\to0.
\]
This logarithmic rate is needed because the profile enters the Coulomb phase, which is multiplied by a logarithmic factor in the limiting regimes.

The required mapping and continuity properties of the phase functionals $\mathcal C$, $\mathcal B_L$, and $\mathcal B_\infty$ are established in \Cref{sec:phase-functional-estimates}. In particular, we prove the uniform large-$L$ limit
\[
\mathcal B_L[G] + (\log L)\mathcal C[G] \longrightarrow \mathcal B_\infty[G] \qquad \text{as }L\to\infty,
\]
on the bounded profile class arising in the proof.

The proof of the main theorem is then completed in \Cref{sec:proof-main-theorem}. Combining the frozen-profile asymptotic formula, the quantitative continuity $W_\mu\to W_0$, and the phase-functional estimates, we reduce the asymptotic analysis to the frozen identity
\[
\Phi_\mu[G](t) = \mathcal C[G]\log t + \mathcal B_{\mu t}[G] - \mathcal B_\mu[G].
\]
If $\mu t\to0$, the transition terms vanish and the Coulomb phase $\mathcal C[W_0]\log t$ remains. If $\mu t\to L\in(0,\infty)$, the finite correction $\mathcal B_L[W_0]$ appears. If $\mu t\to\infty$, the logarithmic time phase saturates at $\log(1/\mu)$, and the limiting correction $\mathcal B_\infty[W_0]$ appears. This gives the three asymptotic regimes and completes the proof of \Cref{thm:main-trichotomy}.

\section{Uniform estimates and rescaled dynamics}
\label{sec:uniform-estimates-rescaled-dynamics}

We establish estimates uniform in $0\leq\mu\leq1$ and then derive the rescaled equation and the frozen-phase identity.

\subsection{Uniform estimates}
\label{sec:preliminary-estimates}

\begin{proposition}
\label{prop:uniform-H1-bound-both-signs}
There exists $\varepsilon_0>0$ such that, if
\[
\norm{u_{\mathrm{in}}}_{H^1}\leq\varepsilon \leq \varepsilon_0,
\]
then the corresponding solution satisfies
\begin{equation}\label{eq:uniform-H1-bound-both-signs}
\sup_{0\leq\mu\leq1} \sup_{t\geq0} \norm{u_\mu(t)}_{H^1} \leq C\varepsilon.
\end{equation}
\end{proposition}

\begin{proof}
Mass conservation gives $\norm{u_\mu(t)}_{L^2}=\norm{u_{\mathrm{in}}}_{L^2}$. The quantity
\[
\frac14\norm{\nabla u_\mu(t)}_{L^2}^2+\frac\kappa4\iint_{\R^3\times\R^3}V_\mu(x-y)\abs{u_\mu(t,x)}^2\abs{u_\mu(t,y)}^2\dd x\dd y
\]
is also conserved. Since $0\leq V_\mu(x)\leq\abs{x}^{-1}$, the Hardy--Littlewood--Sobolev inequality, interpolation, and Sobolev embedding give
\begin{equation}\label{eq:hartree-potential-energy-H1-subcritical}
\iint V_\mu(x-y)\abs{u(x)}^2\abs{u(y)}^2\dd x\dd y\leq C\norm{u}_{L^2}^2\norm{u}_{L^3}^2\leq C\norm{u}_{L^2}^3\norm{\nabla u}_{L^2}.
\end{equation}
Conservation of energy and the initial $H^1$ bound therefore imply
\[
\norm{\nabla u_\mu(t)}_{L^2}^2\leq C\varepsilon^2+C\varepsilon^3\norm{\nabla u_\mu(t)}_{L^2}.
\]
Young's inequality and mass conservation give the claim.
\end{proof}

\begin{proposition}
\label{prop:uniform-Hartree-potential-bounds}
For every $\lambda\geq0$, one has
\begin{equation}\label{eq:uniform-potential-Linfty-bound}
\snorm{V_\lambda*\abs{u}^2}_{L^\infty} \leq C \snorm{u}_{L^2}^{4/3} \snorm{u}_{L^\infty}^{2/3},
\end{equation}
and
\begin{equation}\label{eq:uniform-gradient-potential-Linfty-bound}
\snorm{\nabla V_\lambda*\abs{u}^2}_{L^\infty} \leq C \snorm{u}_{L^2}^{2/3} \snorm{u}_{L^\infty}^{4/3}.
\end{equation}
\end{proposition}

These estimates are uniform in $\lambda\geq0$. To prove them, split the convolution into $\abs{x-y}\leq R$ and $\abs{x-y}>R$. The bounds $V_\lambda(x)\leq\abs{x}^{-1}$ and $\abs{\nabla V_\lambda(x)}\leq C\abs{x}^{-2}$ give
\[
\begin{aligned}
\snorm{V_\lambda*\abs{u}^2}_{L^\infty}&\leq C\bigl(R^2\norm{u}_{L^\infty}^2+R^{-1}\norm{u}_{L^2}^2\bigr),\\
\snorm{\nabla V_\lambda*\abs{u}^2}_{L^\infty}&\leq C\bigl(R\norm{u}_{L^\infty}^2+R^{-2}\norm{u}_{L^2}^2\bigr).
\end{aligned}
\]
For $u\neq0$, take $R=(\norm{u}_{L^2}/\norm{u}_{L^\infty})^{2/3}$. The case $u=0$ is immediate.

\begin{proposition}
\label{prop:uniform-HN-estimates}
Under the assumptions of \Cref{thm:main-trichotomy}, one has
\begin{equation}\label{eq:uniform-HN-Linfty-decay}
\sup_{0\leq\mu\leq1} \norm{u_\mu(t)}_{L_x^\infty} \leq C\varepsilon t^{-3/2}, \qquad t\geq1.
\end{equation}
Moreover,
\begin{equation}\label{eq:uniform-HN-Sigma-growth}
\sup_{0\leq\mu\leq1} \left( \snorm{u_\mu(t)}_{H_x^4} + \snorm{\abs{x}^4 f_\mu(t)}_{L_x^2} \right) \leq C\varepsilon t^{C\varepsilon^2}, \qquad t\geq1,
\end{equation}
where
\begin{equation}\label{eq:Schrodinger-profile-definition}
f_\mu(t) := \ee^{-\frac{\ii t}{2}\Delta}u_\mu(t).
\end{equation}
The constants are independent of $0\leq\mu\leq1$.
\end{proposition}

\begin{proof}
We use the a priori argument of Hayashi--Naumkin \cite[Section 4]{HayashiNaumkin1998AJM}.

Let
\[
U(t):=\ee^{\frac{\ii t}{2}\Delta}, \qquad J(t):=x+\ii t\nabla = U(t)xU(-t).
\]
The operator $J(t)$ commutes with $\ii\partial_t+\frac12\Delta$.
Furthermore,
\[
x^\alpha f_\mu(t) = U(-t)J^\alpha(t)u_\mu(t),
\]
and therefore
\begin{equation}\label{eq:J-weight-equivalence}
\snorm{\abs{x}^4 f_\mu(t)}_{L^2} \leq \sum_{\abs{\alpha}\leq4} \snorm{J^\alpha(t)u_\mu(t)}_{L^2}.
\end{equation}

The estimates in \cite[Section 4]{HayashiNaumkin1998AJM} are based on the Coulomb Hartree bounds for the operator $u\mapsto (\abs{x}^{-1}*\abs{u}^2)u$. They apply uniformly to the Yukawa kernels because
\[
0\leq V_\mu(x) = \frac{\ee^{-\mu\abs{x}}}{\abs{x}} \leq \frac1{\abs{x}}, \qquad 0\leq\mu\leq1.
\]
For differentiated convolutions, derivatives are transferred to the density in the sense of distributions,
\[
\partial^\alpha(V_\mu*\abs{u}^2) =V_\mu*\partial^\alpha(\abs{u}^2).
\]
The Hardy--Littlewood--Sobolev and product estimates in \cite[Section 4]{HayashiNaumkin1998AJM} therefore use the undifferentiated kernel, which is bounded by $\abs{x}^{-1}$ uniformly in $\mu$. For the weighted estimates, conjugation by multiplication by $\ee^{\ii\abs{x}^2/(2t)}$ transforms $J(t)$ into $\ii t\nabla$ and preserves the density $\abs{u}^2$. The same product estimates thus give the corresponding $J^\alpha(t)$ bounds with constants independent of $\mu$.

Applying $\partial^\alpha$, $\abs{\alpha}\leq4$, to the equation and using the corresponding Hartree estimates gives
\begin{equation}\label{eq:HN-H4-differential}
\frac{\dd}{\dd t} \norm{u_\mu(t)}_{H^4} \leq C\norm{u_\mu(t)}_{L^\infty}^{2/3} \norm{u_\mu(t)}_{L^2}^{4/3} \norm{u_\mu(t)}_{H^4}.
\end{equation}
Similarly, since $J(t)$ commutes with the free Schr\"odinger operator, the same argument applied to $J^\alpha(t)u_\mu$ gives
\begin{equation}\label{eq:HN-J-differential}
\frac{\dd}{\dd t} \sum_{\abs{\alpha}\leq4} \norm{J^\alpha(t)u_\mu(t)}_{L^2} \leq C\norm{u_\mu(t)}_{L^\infty}^{2/3} \norm{u_\mu(t)}_{L^2}^{4/3} \sum_{\abs{\alpha}\leq4} \norm{J^\alpha(t)u_\mu(t)}_{L^2}.
\end{equation}
Here the factor
\[
\norm{u_\mu(t)}_{L^\infty}^{2/3} \norm{u_\mu(t)}_{L^2}^{4/3}
\]
is the uniform bound for $\snorm{V_\mu*\abs{u_\mu}^2}_{L^\infty}$ in \Cref{prop:uniform-Hartree-potential-bounds}.

The Hayashi--Naumkin bootstrap gives, uniformly in $0\leq\mu\leq1$,
\begin{equation}\label{eq:HN-bootstrap-Linfty-input}
\norm{u_\mu(t)}_{L^\infty} \leq C\varepsilon t^{-3/2}, \qquad t\geq1.
\end{equation}
Substituting this into \eqref{eq:HN-H4-differential} and \eqref{eq:HN-J-differential}, and using mass conservation, we obtain
\[
\frac{\dd}{\dd t} \Bigl( \norm{u_\mu(t)}_{H^4} + \sum_{\abs{\alpha}\leq4} \norm{J^\alpha(t)u_\mu(t)}_{L^2} \Bigr) \leq C\varepsilon^2 t^{-1} \Bigl( \norm{u_\mu(t)}_{H^4} + \sum_{\abs{\alpha}\leq4} \norm{J^\alpha(t)u_\mu(t)}_{L^2} \Bigr).
\]
By Gr\"onwall's inequality,
\[
\norm{u_\mu(t)}_{H^4} + \sum_{\abs{\alpha}\leq4} \norm{J^\alpha(t)u_\mu(t)}_{L^2} \leq C \Bigl( \norm{u_\mu(1)}_{H^4} + \sum_{\abs{\alpha}\leq4} \norm{J^\alpha(1)u_\mu(1)}_{L^2} \Bigr) t^{C\varepsilon^2}.
\]
The local theory and the smallness of the initial datum in $\Sigma_4$ imply
\[
\sup_{0\leq\mu\leq1} \Bigl( \norm{u_\mu(1)}_{H^4} + \sum_{\abs{\alpha}\leq4} \norm{J^\alpha(1)u_\mu(1)}_{L^2} \Bigr) \leq C\varepsilon.
\]
Combining this with \eqref{eq:J-weight-equivalence} proves \eqref{eq:uniform-HN-Sigma-growth}.

To obtain \eqref{eq:HN-bootstrap-Linfty-input}, the factorization used in \cite[Section 4]{HayashiNaumkin1998AJM}, with the unitary Fourier transform, gives
\[
\begin{aligned}
u_\mu(t,x)={}&(\ii t)^{-3/2}\ee^{\ii\abs{x}^2/(2t)}\widehat{f_\mu(t)}(x/t)\\
&+(\ii t)^{-3/2}\ee^{\ii\abs{x}^2/(2t)}\mathcal F\bigl((\ee^{\ii\abs{\,\cdot\,}^2/(2t)}-1)f_\mu(t)\bigr)(x/t).
\end{aligned}
\]
The phase-corrected Fourier-profile argument in that reference applies with resonant coefficient $V_{\mu t}*|\widehat{f_\mu(t)}|^2$. 
The Hartree estimates above control the errors uniformly in $\mu$ and give
\[
\sup_{0\leq\mu\leq1}\snorm{\widehat{f_\mu(t)}}_{L^\infty}\leq C\varepsilon,\qquad t\geq1.
\]
For the factorization error, $\abs{\ee^{\ii r}-1}\leq\abs{r}$ and Cauchy--Schwarz give
\[
\begin{aligned}
\norm{\mathcal F\bigl((\ee^{\ii\abs{\,\cdot\,}^2/(2t)}-1)f_\mu(t)\bigr)}_{L^\infty} &\leq Ct^{-1}\snorm{\abs{x}^2f_\mu(t)}_{L^1}\\
&\leq Ct^{-1}\bigl(\norm{f_\mu(t)}_{L^2}+\snorm{\abs{x}^4f_\mu(t)}_{L^2}\bigr)\\
&\leq C\varepsilon t^{-1+C\varepsilon^2}.
\end{aligned}
\]
Here $\abs{x}^2/(1+\abs{x}^4)\in L^2(\R^3)$. Consequently,
\[
\norm{u_\mu(t)}_{L^\infty}\leq Ct^{-3/2}\snorm{\widehat{f_\mu(t)}}_{L^\infty}+C\varepsilon t^{-5/2+C\varepsilon^2}\leq C\varepsilon t^{-3/2}
\]
for sufficiently small $\varepsilon_0$. This proves \eqref{eq:uniform-HN-Linfty-decay} and closes the bootstrap.
\end{proof}

\subsection{Rescaled dynamics and frozen phase}
\label{sec:rescaled-amplitude-frozen-phase}
\label{subsec:definition-rescaled-amplitude}
\label{subsec:derivation-exact-amplitude-equation}
\label{subsec:dispersive-error}
\label{subsec:frozen-phase-identity}

For $t>0$, define the rescaled amplitude by
\begin{equation}\label{eq:main-rescaled-amplitude-definition}
a_\mu(t,v) := (\ii t)^{3/2} \ee^{-\frac{\ii t}{2}\abs{v}^2} u_\mu(t,tv).
\end{equation}
Equivalently,
\begin{equation}\label{eq:main-u-from-amplitude}
u_\mu(t,x) = (\ii t)^{-3/2} \ee^{\frac{\ii\abs{x}^2}{2t}} a_\mu(t,x/t).
\end{equation}
The map $a\mapsto u$ in \eqref{eq:main-u-from-amplitude} is unitary from $L_v^2$ to $L_x^2$.

\begin{corollary}
\label{cor:weighted-profile-amplitude-consequences}
The weighted-profile estimate implies the corresponding $v$-regularity control of the rescaled amplitude. In particular, for $\abs{\alpha}\leq4$,
\begin{equation}\label{eq:amplitude-derivative-growth-bound}
\sup_{0\leq\mu\leq1} \norm{\partial_v^\alpha a_\mu(t)}_{L_v^2} \leq C\varepsilon t^{C\varepsilon^2}, \qquad t\geq1.
\end{equation}
\end{corollary}

\begin{proof}
The free Schr\"odinger kernel and the unitary Fourier convention give
\[
a_\mu(t)=\mathcal F\bigl(\ee^{\ii\abs{\,\cdot\,}^2/(2t)}f_\mu(t)\bigr) =\ee^{-\ii\Delta_v/(2t)}\widehat{f_\mu(t)}.
\]
Plancherel's identity therefore yields $\snorm{\partial_v^\alpha a_\mu(t)}_{L_v^2}=\snorm{x^\alpha f_\mu(t)}_{L_x^2}$. For $\abs{\alpha}\leq4$, the right-hand side is bounded by $C(\snorm{f_\mu(t)}_{L^2}+\snorm{\abs{x}^4f_\mu(t)}_{L^2})$. Mass conservation and \eqref{eq:uniform-HN-Sigma-growth} prove the claim.
\end{proof}

Set $v=x/t$. A direct computation gives
\begin{equation}\label{eq:linear-part-amplitude-computation}
\left(\ii\partial_t+\frac12\Delta_x\right)u_\mu(t,x) = (\ii t)^{-3/2} \ee^{\frac{\ii\abs{x}^2}{2t}} \left( \ii\partial_t a_\mu(t,v) + \frac1{2t^2}\Delta_v a_\mu(t,v) \right).
\end{equation}
On the other hand,
\[
\abs{u_\mu(t,x)}^2 = t^{-3} \abs{a_\mu(t,x/t)}^2.
\]
Changing variables $y=tw$, we obtain
\begin{equation}\label{eq:rescaled-nonlinear-potential}
\left(V_\mu*\abs{u_\mu(t)}^2\right)(tv) = \frac1t \left(V_{\mu t}*\abs{a_\mu(t)}^2\right)(v).
\end{equation}
Substituting \eqref{eq:linear-part-amplitude-computation} and \eqref{eq:rescaled-nonlinear-potential} into \eqref{eq:hartree-yukawa-coulomb}, we obtain the exact amplitude equation
\begin{equation}\label{eq:exact-rescaled-amplitude-equation}
\partial_t a_\mu(t,v) = \frac{\ii}{2t^2}\Delta_v a_\mu(t,v) - \frac{\ii\kappa}{t} \left(V_{\mu t}*\abs{a_\mu(t)}^2\right)(v) a_\mu(t,v).
\end{equation}
The nonlinear kernel depends on $\mu$ and $t$ through the product $\mu t$.
The dispersive term in \eqref{eq:exact-rescaled-amplitude-equation} is integrable in time.

\begin{lemma}\label{lem:integrability-dispersive-error}
Under the assumptions of \Cref{thm:main-trichotomy}, one has
\begin{equation}\label{eq:dispersive-error-integrability-L2}
\sup_{0\leq\mu\leq1}\int_1^\infty t^{-2}\norm{\Delta_va_\mu(t)}_{L_v^2}\dd t\leq C\varepsilon
\end{equation}
and
\begin{equation}\label{eq:dispersive-error-integrability-Linfty}
\sup_{0\leq\mu\leq1}\int_1^\infty t^{-2}\norm{\Delta_va_\mu(t)}_{L_v^\infty}\dd t\leq C\varepsilon.
\end{equation}
\end{lemma}

This lemma is the reason why the exact nonlinear phase captures the non-integrable part of the dynamics.

\begin{proof}
By \Cref{cor:weighted-profile-amplitude-consequences} and the embedding $H_v^2\hookrightarrow L_v^\infty$,
\[
\norm{\Delta_va_\mu(t)}_{L_v^2}+\norm{\Delta_va_\mu(t)}_{L_v^\infty} \leq C\norm{a_\mu(t)}_{H_v^4}\leq C\varepsilon t^{C\varepsilon^2},\qquad t\geq1.
\]
Choose $\varepsilon_0$ so that $C\varepsilon_0^2<1/2$. Both claims follow by integrating $t^{-2+C\varepsilon^2}$ over $[1,\infty)$.
\end{proof}

\begin{remark}\label{rem:weighted-profile-controls-dispersive-error}
The identity in the proof of \Cref{cor:weighted-profile-amplitude-consequences} transfers spatial moments of $f_\mu$ to derivatives of $a_\mu$. Thus the weighted estimate \eqref{eq:uniform-HN-Sigma-growth} gives the time integrability of $t^{-2}\Delta_va_\mu(t)$.
\end{remark}

\begin{corollary}
\label{cor:uniform-yukawa-potential-bounds}
Under the assumptions of \Cref{thm:main-trichotomy}, for every $t\geq1$ and $0<\mu\leq1$,
\begin{equation}\label{eq:uniform-yukawa-potential-decay}
\norm{V_\mu*\abs{u_\mu(t)}^2}_{L^\infty} \leq C\varepsilon^2 \min\left\{t^{-1},\mu^{-2}t^{-3}\right\}.
\end{equation}
The bound $C\varepsilon^2t^{-1}$ also holds for $\mu=0$.
\end{corollary}

\begin{proof}
By mass conservation, \eqref{eq:uniform-HN-Linfty-decay}, and \eqref{eq:uniform-potential-Linfty-bound},
\[
\begin{aligned}
\norm{V_\mu*\abs{u_\mu(t)}^2}_{L^\infty} &\leq C\norm{u_\mu(t)}_{L^2}^{4/3} \norm{u_\mu(t)}_{L^\infty}^{2/3}\\
&\leq C\varepsilon^{4/3}(\varepsilon t^{-3/2})^{2/3} =C\varepsilon^2t^{-1}.
\end{aligned}
\]
For $\mu>0$, Young's inequality and $\norm{V_\mu}_{L^1}=4\pi\mu^{-2}$ give
\[
\norm{V_\mu*\abs{u_\mu(t)}^2}_{L^\infty} \leq \norm{V_\mu}_{L^1}\norm{u_\mu(t)}_{L^\infty}^2 \leq C\mu^{-2}\varepsilon^2t^{-3}.
\]
\end{proof}

The frozen phase admits the following explicit decomposition.

\begin{lemma}
\label{lem:frozen-phase-decomposition}
For every sufficiently regular profile $G$, one has
\begin{equation}\label{eq:frozen-phase-decomposition}
\Phi_\mu[G](t) = \cC[G]\log t + \cB_{\mu t}[G] - \cB_\mu[G].
\end{equation}
\end{lemma}

\begin{proof}
For $\mu=0$, the identity follows directly from $b_0=0$. We therefore assume $\mu>0$.

By definition,
\[
\Phi_\mu[G](t,v) = \kappa \int_{\R^3} \left[ \int_1^t \frac1s \frac{\ee^{-\mu s\abs{v-w}}}{\abs{v-w}} \,\dd s \right] \abs{G(w)}^2 \,\dd w.
\]
For $z=v-w$, write
\[
\int_1^t \frac1s \frac{\ee^{-\mu s\abs{z}}}{\abs{z}} \,\dd s = \frac{\log t}{\abs{z}} + \frac1{\abs{z}} \int_1^t \frac{\ee^{-\mu s\abs{z}}-1}{s} \,\dd s.
\]
Changing variables $r=\mu s$ in the second integral gives
\[
\frac1{\abs{z}} \int_\mu^{\mu t} \frac{\ee^{-r\abs{z}}-1}{r} \,\dd r = b_{\mu t}(z)-b_\mu(z).
\]
This proves \eqref{eq:frozen-phase-decomposition}.
\end{proof}

\section{Asymptotic profiles}
\label{sec:exact-renormalization-asymptotic-profiles}

We construct the exact renormalized profile $Z_\mu$, apply the finite Yukawa phase correction to obtain $W_\mu$, and prove the frozen-profile asymptotic formula.

\subsection{Exact renormalization}
\label{subsec:exact-renormalized-amplitude}

\begin{definition}[Exact nonlinear phase and renormalized amplitude]
\label{def:exact-renormalization}
Define
\begin{equation}\label{eq:exact-nonlinear-phase-definition}
\Theta_\mu(t,v) := \kappa \int_1^t \frac1s \bigl(V_{\mu s}*\abs{a_\mu(s)}^2\bigr)(v) \,\dd s,
\end{equation}
and
\begin{equation}\label{eq:renormalized-amplitude-definition}
b_\mu(t,v) := \ee^{\ii\Theta_\mu(t,v)}a_\mu(t,v).
\end{equation}
\end{definition}
This removes the exact long-range phase generated by the evolving density $\abs{a_\mu(s)}^2$.

\begin{proposition}
\label{prop:exact-renormalized-profile}
Under the assumptions of \Cref{thm:main-trichotomy}, there exists
\[
Z_\mu\in L^2(\R^3)\cap L^\infty(\R^3)
\]
such that
\begin{equation}\label{eq:exact-profile-convergence}
b_\mu(t) \to Z_\mu
\end{equation}
in $L_v^2\cap L_v^\infty$ as $t\to\infty$. Moreover, the convergence has the tail estimate
\begin{equation}\label{eq:exact-profile-tail-bound}
\norm{b_\mu(t)-Z_\mu}_{L_v^2\cap L_v^\infty} \leq C\varepsilon t^{-1+C\varepsilon^2}, \qquad t\geq1,
\end{equation}
with $C$ independent of $0\leq\mu\leq1$.
\end{proposition}

\begin{proof}
By \eqref{eq:exact-rescaled-amplitude-equation} and the definitions of $\Theta_\mu$ and $b_\mu$, the nonlinear terms cancel and
\[
\partial_tb_\mu(t)=\frac{\ii}{2t^2}\ee^{\ii\Theta_\mu(t)}\Delta_va_\mu(t).
\]
Since $\Theta_\mu$ is real-valued, for $p=2,\infty$,
\[
\norm{\partial_tb_\mu(t)}_{L_v^p}=\frac1{2t^2}\norm{\Delta_va_\mu(t)}_{L_v^p}\leq C\varepsilon t^{-2+C\varepsilon^2}.
\]
This bound follows from the proof of \Cref{lem:integrability-dispersive-error} and is integrable for sufficiently small $\varepsilon_0$. Hence $b_\mu(t)$ converges in $L_v^2\cap L_v^\infty$ to a limit $Z_\mu$. Moreover,
\[
\norm{Z_\mu-b_\mu(t)}_{L_v^2\cap L_v^\infty} \leq\int_t^\infty\norm{\partial_sb_\mu(s)}_{L_v^2\cap L_v^\infty}\dd s \leq C\varepsilon t^{-1+C\varepsilon^2}.
\]
The same estimate and $b_\mu(1)=a_\mu(1)$ give $\sup_{t\geq1}\norm{b_\mu(t)}_{L_v^2\cap L_v^\infty}+\norm{Z_\mu}_{L_v^2\cap L_v^\infty}\leq C\varepsilon$, uniformly in $\mu$.
\end{proof}

\begin{proposition}
\label{prop:Zmu-tail-bound}
Under the assumptions of \Cref{thm:main-trichotomy}, the exact renormalized profile satisfies
\[
Z_\mu\in H_v^1(\R^3)
\]
and obeys the uniform tail estimate
\begin{equation}\label{eq:Zmu-tail-H1-bound}
\norm{Z_\mu-b_\mu(t)}_{H_v^1} \leq C\varepsilon t^{-1+C\varepsilon^2}, \qquad t\geq1, \qquad 0\leq\mu\leq1.
\end{equation}
\end{proposition}
This estimate is used to compare the exact renormalized profile with frozen-profile asymptotics and to prove continuity as $\mu\to0$.

\begin{proof}
Differentiating the identity for $\partial_tb_\mu$ gives
\[
\nabla_v\partial_tb_\mu(t)=\frac{\ii}{2t^2}\ee^{\ii\Theta_\mu(t)} \bigl(\nabla_v\Delta_va_\mu(t)+\ii\nabla_v\Theta_\mu(t)\Delta_va_\mu(t)\bigr).
\]
Consequently,
\begin{equation}\label{eq:bt-H1-derivative-bound-preliminary}
\norm{\partial_tb_\mu(t)}_{H_v^1} \leq Ct^{-2}\bigl(\norm{a_\mu(t)}_{H_v^3}+\norm{\nabla_v\Theta_\mu(t)}_{L_v^\infty}\norm{\Delta_va_\mu(t)}_{L_v^2}\bigr).
\end{equation}
By \Cref{cor:weighted-profile-amplitude-consequences},
\begin{equation}\label{eq:disp-error-H1-pointwise}
t^{-2}\norm{a_\mu(t)}_{H_v^3}\leq C\varepsilon t^{-2+C\varepsilon^2}.
\end{equation}
The amplitude transformation and the decay estimate imply
\[
\norm{a_\mu(s)}_{L_v^2}=\norm{u_\mu(s)}_{L_x^2}\leq\varepsilon, \qquad \norm{a_\mu(s)}_{L_v^\infty}=s^{3/2}\norm{u_\mu(s)}_{L_x^\infty}\leq C\varepsilon.
\]
Using \eqref{eq:uniform-gradient-potential-Linfty-bound} for $V_{\mu s}$, we obtain
\[
\begin{aligned}
\norm{\nabla_v\Theta_\mu(t)}_{L_v^\infty} &\leq\int_1^t\norm{\nabla V_{\mu s}*\abs{a_\mu(s)}^2}_{L_v^\infty}\frac{\dd s}{s}\\
&\leq C\int_1^t\norm{a_\mu(s)}_{L_v^2}^{2/3}\norm{a_\mu(s)}_{L_v^\infty}^{4/3}\frac{\dd s}{s} \leq C\varepsilon^2\log t.
\end{aligned}
\]
Since $\varepsilon^2\log t\leq t^{\varepsilon^2}$ for $t\geq1$, \eqref{eq:bt-H1-derivative-bound-preliminary} gives, after increasing $C$ in the exponent,
\[
\norm{\partial_tb_\mu(t)}_{H_v^1}\leq C\varepsilon t^{-2+C\varepsilon^2}.
\]
For sufficiently small $\varepsilon_0$, $b_\mu(t)$ is therefore Cauchy in $H_v^1$. Its limit agrees with $Z_\mu$ by the $L_v^2$ convergence already proved. Integrating the last estimate from $t$ to infinity proves \eqref{eq:Zmu-tail-H1-bound}.
\end{proof}

\subsection{Frozen-profile asymptotics}
\label{subsec:yukawa-correction-profile-Wmu}
\label{sec:frozen-profile-asymptotics}

The exact profile $Z_\mu$ is not yet the profile which appears in the explicit frozen phase. One must absorb the finite phase error produced by replacing $\abs{a_\mu(s)}^2$ by $\abs{Z_\mu}^2$.

\begin{definition}[Yukawa correction and $W_\mu$]
\label{def:yukawa-correction-Wmu}
Define the Yukawa phase correction by
\begin{equation}\label{eq:GammaY-definition}
\Gamma_\mu^Y(v) := \kappa \int_1^\infty \frac1s \left( V_{\mu s}* \Bigl( \abs{a_\mu(s)}^2 - \abs{Z_\mu}^2 \Bigr) \right)(v) \,\dd s.
\end{equation}
Then define the frozen-profile normalized asymptotic profile by
\begin{equation}\label{eq:Wmu-definition}
W_\mu = \ee^{-\ii\Gamma_\mu^Y}Z_\mu.
\end{equation}
\end{definition}

The integral is absolutely convergent in $L_v^\infty$ by the tail estimate for $b_\mu(t)-Z_\mu$ and the convolution bound \eqref{eq:K-convolution-L1-Linfty-bound}.

\begin{remark}
The correction $\Gamma_\mu^Y$ converts the exact profile $Z_\mu$, obtained from the evolving density $\abs{a_\mu(s)}^2$, into the frozen-profile normalization used in \Cref{prop:frozen-profile-asymptotics}.
\end{remark}

We now replace the exact evolving phase by the frozen phase generated by $W_\mu$.

\begin{proposition}
\label{prop:frozen-profile-asymptotics}
Under the assumptions of \Cref{thm:main-trichotomy}, one has
\begin{equation}\label{eq:frozen-profile-asymptotic-quantitative}
\sup_{0\leq\mu\leq1} \norm{ a_\mu(t) - \ee^{-\ii\Phi_\mu[W_\mu](t)}W_\mu }_{L_v^2} \leq C\varepsilon t^{-1+C\varepsilon^2}, \qquad t\geq1.
\end{equation}
In particular, for each fixed $0\leq\mu\leq1$,
\begin{equation}\label{eq:frozen-profile-asymptotic}
\norm{ a_\mu(t) - \ee^{-\ii\Phi_\mu[W_\mu](t)}W_\mu }_{L_v^2} \longrightarrow0 \qquad \text{as }t\to\infty.
\end{equation}
The uniform bound also applies along sequences $\mu_n\to0$ and $t_n\to\infty$.
\end{proposition}
We use this uniform approximation in each of the three regimes.

\begin{proof}
We use the exact renormalized amplitude
\[
b_\mu(t) = \ee^{\ii\Theta_\mu(t)}a_\mu(t), \qquad b_\mu(t)\to Z_\mu
\]
from \Cref{prop:exact-renormalized-profile}. Since $\Theta_\mu$ in \eqref{eq:exact-nonlinear-phase-definition} is real-valued, we have
\[
\abs{a_\mu(t)} = \abs{b_\mu(t)}.
\]
Decomposing the density as
\[
\abs{a_\mu(s)}^2 = \abs{Z_\mu}^2 + \left( \abs{a_\mu(s)}^2 - \abs{Z_\mu}^2 \right)
\]
and using \eqref{eq:GammaY-definition}, we have
\begin{equation}\label{eq:Theta-frozen-Gamma-tail-decomposition}
\Theta_\mu(t) = \Phi_\mu[Z_\mu](t) + \Gamma_\mu^Y - R_\mu(t),
\end{equation}
where the remaining tail is
\begin{equation}\label{eq:Rmu-tail-definition}
R_\mu(t,v) := \kappa \int_t^\infty \frac1s \left( V_{\mu s}* \left( \abs{a_\mu(s)}^2 - \abs{Z_\mu}^2 \right) \right)(v) \,\dd s.
\end{equation}

We first estimate $R_\mu(t)$. We use the elementary convolution bound
\begin{equation}\label{eq:K-convolution-L1-Linfty-bound}
\norm{V_{\lambda}*h}_{L^\infty} \leq C\norm{h}_{L^1}^{2/3} \norm{h}_{L^\infty}^{1/3}, \qquad \lambda\geq0.
\end{equation}
Since $V_\lambda(z)\leq\abs{z}^{-1}$, for every $R>0$,
\[
\abs{(V_\lambda*h)(v)} \leq C R^2\norm{h}_{L^\infty} + R^{-1}\norm{h}_{L^1},
\]
and choosing $R =(\norm{h}_{L^1}/\norm{h}_{L^\infty})^{1/3}$, we obtain \eqref{eq:K-convolution-L1-Linfty-bound}.

Let
\[
h_\mu(s) := \abs{a_\mu(s)}^2-\abs{Z_\mu}^2.
\]
Since $\abs{a_\mu(s)}=\abs{b_\mu(s)}$, we have
\[
h_\mu(s) = \abs{b_\mu(s)}^2-\abs{Z_\mu}^2.
\]
By H\"older's inequality,
\[
\norm{h_\mu(s)}_{L^1} \leq \left( \norm{b_\mu(s)}_{L^2} + \norm{Z_\mu}_{L^2} \right) \norm{b_\mu(s)-Z_\mu}_{L^2},
\]
and
\[
\norm{h_\mu(s)}_{L^\infty} \leq \left( \norm{b_\mu(s)}_{L^\infty} + \norm{Z_\mu}_{L^\infty} \right) \norm{b_\mu(s)-Z_\mu}_{L^\infty}.
\]
By \Cref{prop:exact-renormalized-profile},
\[
\norm{b_\mu(s)-Z_\mu}_{L^2\cap L^\infty} \leq C\varepsilon s^{-1+C\varepsilon^2},
\]
and together with the uniform $L^2$ and $L^\infty$ bounds,
\[
\norm{b_\mu(s)}_{L^2\cap L^\infty} + \norm{Z_\mu}_{L^2\cap L^\infty} \leq C\varepsilon.
\]
Hence
\[
\norm{h_\mu(s)}_{L^1} + \norm{h_\mu(s)}_{L^\infty} \leq C\varepsilon^2s^{-1+C\varepsilon^2}.
\]
Using \eqref{eq:K-convolution-L1-Linfty-bound}, we obtain
\[
\norm{V_{\mu s}*h_\mu(s)}_{L^\infty} \leq C\varepsilon^2s^{-1+C\varepsilon^2}.
\]
Thus
\begin{equation}
\norm{R_\mu(t)}_{L^\infty} \leq C \int_t^\infty \frac1s \norm{V_{\mu s}*h_\mu(s)}_{L^\infty} \,\dd s \leq C\varepsilon^2 \int_t^\infty s^{-2+C\varepsilon^2} \,\dd s \leq C\varepsilon^2t^{-1+C\varepsilon^2}. \label{eq:Rmu-tail-Linfty-bound}
\end{equation}

Now, since $\abs{W_\mu}=\abs{Z_\mu}$,
\begin{equation}\label{eq:Phi-Wmu-equals-Phi-Zmu}
\Phi_\mu[W_\mu](t) = \Phi_\mu[Z_\mu](t).
\end{equation}
Using $a_\mu(t)=\ee^{-\ii\Theta_\mu(t)}b_\mu(t)$, \eqref{eq:Theta-frozen-Gamma-tail-decomposition}, and \eqref{eq:Phi-Wmu-equals-Phi-Zmu}, we get
\[
a_\mu(t) = \ee^{-\ii\Theta_\mu(t)}b_\mu(t) = \ee^{-\ii\Phi_\mu[Z_\mu](t)} \ee^{-\ii\Gamma_\mu^Y} \ee^{\ii R_\mu(t)} b_\mu(t) = \ee^{-\ii\Phi_\mu[W_\mu](t)} \ee^{\ii R_\mu(t)} \ee^{-\ii\Gamma_\mu^Y} b_\mu(t).
\]
On the other hand,
\[
\ee^{-\ii\Phi_\mu[W_\mu](t)}W_\mu = \ee^{-\ii\Phi_\mu[W_\mu](t)} \ee^{-\ii\Gamma_\mu^Y} Z_\mu.
\]
Since multiplication by a complex exponential of real phase is unitary on $L^2$, we obtain
\[
\norm{ a_\mu(t) - \ee^{-\ii\Phi_\mu[W_\mu](t)}W_\mu }_{L^2} = \norm{ \ee^{\ii R_\mu(t)}b_\mu(t) - Z_\mu }_{L^2} \leq \norm{b_\mu(t)-Z_\mu}_{L^2} + \norm{ \left( \ee^{\ii R_\mu(t)}-1 \right) b_\mu(t) }_{L^2}.
\]
Using $\abs{\ee^{\ii r}-1}\leq C\abs{r}$, \eqref{eq:Rmu-tail-Linfty-bound}, and $\norm{b_\mu(t)}_{L^2}\leq C\varepsilon$, we get
\[
\norm{ \left( \ee^{\ii R_\mu(t)}-1 \right) b_\mu(t) }_{L^2} \leq C\norm{R_\mu(t)}_{L^\infty} \norm{b_\mu(t)}_{L^2} \leq C\varepsilon^3t^{-1+C\varepsilon^2}.
\]
Also, by \Cref{prop:exact-renormalized-profile},
\[
\norm{b_\mu(t)-Z_\mu}_{L^2} \leq C\varepsilon t^{-1+C\varepsilon^2}.
\]
Therefore, combining the last two estimates gives \eqref{eq:frozen-profile-asymptotic-quantitative}. Since all constants are uniform in $0\leq\mu\leq1$, we also obtain \eqref{eq:frozen-profile-asymptotic}.
\end{proof}

\begin{corollary}
\label{cor:frozen-profile-physical-space}
Under the assumptions of \Cref{thm:main-trichotomy}, one has
\begin{equation}\label{eq:frozen-profile-physical-space-quantitative}
\sup_{0\leq\mu\leq1} \norm{ u_\mu(t,x) - (\ii t)^{-3/2} \ee^{\frac{\ii\abs{x}^2}{2t}} \ee^{-\ii\Phi_\mu[W_\mu](t,x/t)} W_\mu(x/t) }_{L_x^2} \leq C\varepsilon t^{-1+C\varepsilon^2}.
\end{equation}
In particular,
\begin{equation}\label{eq:frozen-profile-physical-space}
u_\mu(t,x) = (\ii t)^{-3/2} \ee^{\frac{\ii\abs{x}^2}{2t}} \ee^{-\ii\Phi_\mu[W_\mu](t,x/t)} W_\mu(x/t) + o_{L_x^2}(1)
\end{equation}
as $t\to\infty$, uniformly in $0\leq\mu\leq1$.
\end{corollary}

\begin{proof}
By \eqref{eq:main-u-from-amplitude},
\[
u_\mu(t,x) = (\ii t)^{-3/2} \ee^{\frac{\ii\abs{x}^2}{2t}} a_\mu(t,x/t).
\]
The map
\[
A(v) \mapsto (\ii t)^{-3/2} \ee^{\frac{\ii\abs{x}^2}{2t}} A(x/t)
\]
is unitary from $L_v^2$ to $L_x^2$. Therefore
\[
\norm{ u_\mu(t,x) - (\ii t)^{-3/2} \ee^{\frac{\ii\abs{x}^2}{2t}} \ee^{-\ii\Phi_\mu[W_\mu](t,x/t)} W_\mu(x/t) }_{L_x^2} = \norm{ a_\mu(t,v) - \ee^{-\ii\Phi_\mu[W_\mu](t,v)} W_\mu(v) }_{L_v^2}.
\]
The claim follows directly from \Cref{prop:frozen-profile-asymptotics}.
\end{proof}

\section{Quantitative continuity of the profiles}
\label{sec:quantitative-continuity-profiles}

This section proves that the asymptotic profiles converge strongly enough as $\mu\to0$ to absorb the large logarithmic phase.

\subsection{Continuity of the exact profiles}
\label{subsec:continuity-exact-profiles}

\begin{lemma}
\label{lem:short-time-continuity-mu}
Under the assumptions of \Cref{thm:main-trichotomy}, one has
\begin{equation}\label{eq:short-time-continuity-mu}
\sup_{0\leq t\leq1} \norm{u_\mu(t)-u_0(t)}_{L_x^2} \leq C\varepsilon^3\mu.
\end{equation}
In particular,
\begin{equation}\label{eq:short-time-continuity-mu-at-one}
\norm{u_\mu(1)-u_0(1)}_{L_x^2} \leq C\varepsilon^3\mu.
\end{equation}
\end{lemma}

\begin{proof}
Let
\[
w_\mu(t):=u_\mu(t)-u_0(t).
\]
Since the initial datum is the same for all $\mu$, we have
\[
w_\mu(0)=0.
\]
By Duhamel's formula and the unitarity of the free Schr\"odinger group on $L^2$, for $0\leq t\leq1$,
\[
\begin{aligned}
\norm{w_\mu(t)}_{L^2} &\leq \int_0^t \norm{ \bigl(V_\mu*\abs{u_\mu(s)}^2\bigr)u_\mu(s) - \bigl(V_0*\abs{u_0(s)}^2\bigr)u_0(s) }_{L^2} \,\dd s.
\end{aligned}
\]
We decompose the nonlinear difference as
\[
\begin{aligned}
&\bigl(V_\mu*\abs{u_\mu}^2\bigr)u_\mu - \bigl(V_0*\abs{u_0}^2\bigr)u_0 \\
&= \bigl(V_\mu*\abs{u_\mu}^2\bigr)(u_\mu-u_0) + \bigl(V_\mu*(\abs{u_\mu}^2-\abs{u_0}^2)\bigr)u_0 + \bigl((V_\mu-V_0)*\abs{u_0}^2\bigr)u_0.
\end{aligned}
\]
For $0\leq s\leq1$, the uniform small-data bounds and Sobolev embedding give
\[
\norm{u_\mu(s)}_{H^1} + \norm{u_0(s)}_{H^1} + \norm{u_\mu(s)}_{L^\infty} + \norm{u_0(s)}_{L^\infty} \leq C\varepsilon.
\]
Using \Cref{prop:uniform-Hartree-potential-bounds}, we get
\[
\norm{ \bigl(V_\mu*\abs{u_\mu(s)}^2\bigr)w_\mu(s) }_{L^2} \leq C\varepsilon^2\norm{w_\mu(s)}_{L^2}.
\]
For the second term, since $V_\mu\leq V_0$, Hardy--Littlewood--Sobolev gives
\[
\norm{V_\mu*(\abs{u_\mu(s)}^2-\abs{u_0(s)}^2)}_{L^6} \leq C\norm{\abs{u_\mu(s)}^2-\abs{u_0(s)}^2}_{L^{6/5}}.
\]
Moreover,
\[
\norm{\abs{u_\mu(s)}^2-\abs{u_0(s)}^2}_{L^{6/5}} \leq C\varepsilon\norm{w_\mu(s)}_{L^2},
\]
and $\norm{u_0(s)}_{L^3}\leq C\varepsilon$. Hence
\[
\norm{ \bigl(V_\mu*(\abs{u_\mu(s)}^2-\abs{u_0(s)}^2)\bigr)u_0(s) }_{L^2} \leq C\varepsilon^2\norm{w_\mu(s)}_{L^2}.
\]
Finally, using
\[
0\leq V_0(x)-V_\mu(x) = \frac{1-\ee^{-\mu\abs{x}}}{\abs{x}} \leq \mu,
\]
we obtain
\[
\norm{(V_\mu-V_0)*\abs{u_0(s)}^2}_{L^\infty} \leq \mu\norm{u_0(s)}_{L^2}^2 \leq C\varepsilon^2\mu.
\]
Therefore
\[
\norm{ \bigl((V_\mu-V_0)*\abs{u_0(s)}^2\bigr)u_0(s) }_{L^2} \leq C\varepsilon^3\mu.
\]
Combining the three estimates gives
\[
\norm{w_\mu(t)}_{L^2} \leq C\varepsilon^2 \int_0^t \norm{w_\mu(s)}_{L^2} \,\dd s + C\varepsilon^3\mu t.
\]
By Gr\"onwall's inequality, for $0\leq t\leq1$,
\[
\norm{w_\mu(t)}_{L^2} \leq C\varepsilon^3\mu t\,\ee^{C\varepsilon^2t} \leq C\varepsilon^3\mu.
\]
This proves \eqref{eq:short-time-continuity-mu}.
\end{proof}

\begin{lemma}
\label{lem:polynomial-finite-time-continuity-mu}
For every $S\geq1$, one has
\begin{equation}\label{eq:polynomial-finite-time-continuity-mu}
\sup_{1\leq t\leq S} \norm{u_\mu(t)-u_0(t)}_{L_x^2} \leq C\varepsilon^3\mu S^{1+C\varepsilon^2}.
\end{equation}
\end{lemma}

\begin{proof}
Let
\[
w_\mu(t):=u_\mu(t)-u_0(t).
\]
By \Cref{lem:short-time-continuity-mu}, we have
\[
\norm{w_\mu(1)}_{L^2} \leq C\varepsilon^3\mu.
\]
For $t\geq1$, Duhamel's formula gives
\[
\norm{w_\mu(t)}_{L^2} \leq \norm{w_\mu(1)}_{L^2} + \int_1^t \norm{ \bigl(V_\mu*\abs{u_\mu(s)}^2\bigr)u_\mu(s) - \bigl(V_0*\abs{u_0(s)}^2\bigr)u_0(s) }_{L^2} \,\dd s.
\]
We use the same decomposition of the nonlinear difference.
\[
\bigl(V_\mu*\abs{u_\mu}^2\bigr)u_\mu - \bigl(V_0*\abs{u_0}^2\bigr)u_0 = \bigl(V_\mu*\abs{u_\mu}^2\bigr)(u_\mu-u_0) + \bigl(V_\mu*(\abs{u_\mu}^2-\abs{u_0}^2)\bigr)u_0 + \bigl((V_\mu-V_0)*\abs{u_0}^2\bigr)u_0.
\]

By \Cref{cor:uniform-yukawa-potential-bounds}, since $s>1$,
\[
\norm{V_\mu*\abs{u_\mu(s)}^2}_{L^\infty} \leq C\varepsilon^2s^{-1}.
\]
Therefore
\[
\norm{ \bigl(V_\mu*\abs{u_\mu(s)}^2\bigr)w_\mu(s) }_{L^2} \leq C\varepsilon^2s^{-1}\norm{w_\mu(s)}_{L^2}.
\]

For the second term, set
\[
h_\mu(s):=\abs{u_\mu(s)}^2-\abs{u_0(s)}^2.
\]
Since $V_\mu\leq V_0$, Hardy--Littlewood--Sobolev gives
\[
\norm{V_\mu*h_\mu(s)}_{L^6} \leq C\norm{h_\mu(s)}_{L^{6/5}}.
\]
By interpolation between $L^2$ and $L^\infty$, the decay estimate gives
\[
\norm{u_\mu(s)}_{L^3} + \norm{u_0(s)}_{L^3} \leq C\varepsilon s^{-1/2}.
\]
Hence
\[
\norm{h_\mu(s)}_{L^{6/5}} \leq C\varepsilon s^{-1/2}\norm{w_\mu(s)}_{L^2}.
\]
Using again $\norm{u_0(s)}_{L^3}\leq C\varepsilon s^{-1/2}$, we get
\[
\norm{ \bigl(V_\mu*h_\mu(s)\bigr)u_0(s) }_{L^2} \leq \norm{V_\mu*h_\mu(s)}_{L^6} \norm{u_0(s)}_{L^3} \leq C\varepsilon^2s^{-1}\norm{w_\mu(s)}_{L^2}.
\]

For the kernel error, since $0\leq V_0(x)-V_\mu(x)\leq\mu$, we have
\[
\norm{(V_\mu-V_0)*\abs{u_0(s)}^2}_{L^\infty} \leq \mu\norm{u_0(s)}_{L^2}^2 \leq \varepsilon^2\mu.
\]
Thus
\[
\norm{ \bigl((V_\mu-V_0)*\abs{u_0(s)}^2\bigr)u_0(s) }_{L^2} \leq C\varepsilon^3\mu.
\]

Combining the estimates, for $1\leq t\leq S$,
\[
\norm{w_\mu(t)}_{L^2} \leq C\varepsilon^3\mu + C\varepsilon^2 \int_1^t s^{-1} \norm{w_\mu(s)}_{L^2} \,\dd s + C\varepsilon^3\mu t.
\]
Therefore
\[
\norm{w_\mu(t)}_{L^2} \leq C\varepsilon^3\mu t + C\varepsilon^2 \int_1^t s^{-1} \norm{w_\mu(s)}_{L^2} \,\dd s.
\]
By Gr\"onwall's inequality,
\[
\norm{w_\mu(t)}_{L^2} \leq C\varepsilon^3\mu t \exp\left( C\varepsilon^2\int_1^t s^{-1}\,\dd s \right) = C\varepsilon^3\mu t^{1+C\varepsilon^2}.
\]
Taking the supremum over $1\leq t\leq S$ proves \eqref{eq:polynomial-finite-time-continuity-mu}.
\end{proof}

\begin{lemma}
\label{lem:continuity-Zmu}
The exact renormalized profiles satisfy
\begin{equation}\label{eq:Zmu-converges-Z0-H1}
\norm{Z_\mu-Z_0}_{H_v^1} \longrightarrow0 \qquad \text{as }\mu\to0.
\end{equation}
\end{lemma}
This lemma follows by combining finite-time continuity with the uniform tail bound \eqref{eq:Zmu-tail-H1-bound}.

\begin{proof}
For $S\geq1$, we write
\[
\norm{Z_\mu-Z_0}_{H_v^1} \leq \norm{Z_\mu-b_\mu(S)}_{H_v^1} + \norm{b_\mu(S)-b_0(S)}_{H_v^1} + \norm{b_0(S)-Z_0}_{H_v^1}.
\]
For the first term and the third term, by \Cref{prop:Zmu-tail-bound},
\[
\norm{Z_\mu-b_\mu(S)}_{H_v^1} + \norm{Z_0-b_0(S)}_{H_v^1} \leq C\varepsilon S^{-1+C\varepsilon^2},
\]
uniformly in $0\leq\mu\leq1$.

All limits in this proof are taken as $\mu\to0$.
It remains to prove that, for every fixed $S\geq1$,
\begin{equation}\label{eq:fixed-time-bmu-continuity-H1}
\norm{b_\mu(S)-b_0(S)}_{H_v^1} \longrightarrow0.
\end{equation}
The amplitude transformation is independent of $\mu$ and isometric on $L^2$. Thus \Cref{lem:polynomial-finite-time-continuity-mu} gives
\[
\norm{a_\mu(S)-a_0(S)}_{L_v^2}=\norm{u_\mu(S)-u_0(S)}_{L_x^2}\longrightarrow0.
\]
The $H_v^4$ bounds from \Cref{cor:weighted-profile-amplitude-consequences} are uniform in $\mu$ for fixed $S$. Interpolation yields
\begin{equation}\label{eq:fixed-time-amplitude-H1-continuity}
\norm{a_\mu(S)-a_0(S)}_{H_v^1} \leq C\norm{a_\mu(S)-a_0(S)}_{L_v^2}^{3/4}\norm{a_\mu(S)-a_0(S)}_{H_v^4}^{1/4}\longrightarrow0.
\end{equation}
Similarly, the embedding $H_v^2\hookrightarrow L_v^\infty$ gives
\begin{equation}\label{eq:fixed-time-amplitude-Linfty-continuity}
\norm{a_\mu(S)-a_0(S)}_{L_v^\infty} \leq C\norm{a_\mu(S)-a_0(S)}_{L_v^2}^{1/2}\norm{a_\mu(S)-a_0(S)}_{H_v^4}^{1/2}\longrightarrow0.
\end{equation}

We next prove convergence of the exact nonlinear phases. Recalling \eqref{eq:exact-nonlinear-phase-definition}, we decompose
\[
\Theta_\mu(S)-\Theta_0(S) = \kappa \int_1^S \frac1s V_{\mu s}* \left( \abs{a_\mu(s)}^2-\abs{a_0(s)}^2 \right) \,\dd s + \kappa \int_1^S \frac1s \left( (V_{\mu s}-V_0)*\abs{a_0(s)}^2 \right) \,\dd s.
\]
For the first term, \eqref{eq:fixed-time-amplitude-H1-continuity}, \eqref{eq:fixed-time-amplitude-Linfty-continuity}, and the uniform bounds imply
\[
\abs{a_\mu(s)}^2-\abs{a_0(s)}^2 \longrightarrow0 \qquad \text{in }L_v^1\cap L_v^\infty
\]
for every fixed $s\in[1,S]$, uniformly with an integrable bound in $s$. Therefore, by
\[
\norm{V_\lambda*h}_{L^\infty} \leq C\norm{h}_{L^1}^{2/3}\norm{h}_{L^\infty}^{1/3} \quad \text{and} \quad \norm{\nabla V_\lambda*h}_{L^\infty} \leq C\norm{h}_{L^1}^{1/3}\norm{h}_{L^\infty}^{2/3},
\]
we obtain
\[
\int_1^S \frac1s \norm{ V_{\mu s}* \left( \abs{a_\mu(s)}^2-\abs{a_0(s)}^2 \right) }_{W_v^{1,\infty}} \,\dd s \longrightarrow0.
\]

For the second term, we use the pointwise convergence $V_{\mu s}\to V_0$ as $\mu\to0$, together with the bounds
\[
0\leq V_0(z)-V_{\mu s}(z) = \frac{1-\ee^{-\mu s\abs{z}}}{\abs{z}} \leq \mu s
\]
and
\[
\norm{(V_{\mu s}-V_0)*\abs{a_0(s)}^2}_{L^\infty} \leq \mu s\norm{a_0(s)}_{L^2}^2.
\]

The corresponding gradient convergence follows by the same near-field/far-field splitting used for the Coulomb gradient kernel. More explicitly, for $R>0$,
\[
\norm{\nabla(V_{\mu s}-V_0)*\abs{a_0(s)}^2}_{L^\infty} \leq C\mu s R^2\norm{a_0(s)}_{L^\infty}^2 + CR^{-2}\norm{a_0(s)}_{L^2}^2.
\]
Choosing $R=(\mu s)^{-1/4}$, and using $1\leq s\leq S$, gives
\[
\norm{\nabla(V_{\mu s}-V_0)*\abs{a_0(s)}^2}_{L^\infty} \leq C_S\varepsilon^2(\mu s)^{1/2} \longrightarrow0.
\]
Therefore
\begin{equation}\label{eq:fixed-time-phase-W1infty-continuity}
\norm{\Theta_\mu(S)-\Theta_0(S)}_{W_v^{1,\infty}} \longrightarrow0 \qquad \text{as }\mu\to0.
\end{equation}

Now recall that
\[
b_\mu(S) = \ee^{\ii\Theta_\mu(S)}a_\mu(S).
\]
Using \eqref{eq:fixed-time-amplitude-H1-continuity} and \eqref{eq:fixed-time-phase-W1infty-continuity}, we obtain
\[
\begin{aligned}
\norm{b_\mu(S)-b_0(S)}_{H_v^1} &= \norm{ \ee^{\ii\Theta_\mu(S)}a_\mu(S) - \ee^{\ii\Theta_0(S)}a_0(S) }_{H_v^1} \\
&= \norm{ \ee^{\ii\Theta_\mu(S)}a_\mu(S) -\ee^{\ii\Theta_\mu(S)}a_0(S) +\ee^{\ii\Theta_\mu(S)}a_0(S) -\ee^{\ii\Theta_0(S)}a_0(S) }_{H_v^1} \\
&\leq \norm{ \ee^{\ii\Theta_\mu(S)}a_\mu(S) -\ee^{\ii\Theta_\mu(S)}a_0(S) }_{H_v^1} + \norm{ \ee^{\ii\Theta_\mu(S)}a_0(S) -\ee^{\ii\Theta_0(S)}a_0(S) }_{H_v^1} \\
&\leq C_S \norm{a_\mu(S)-a_0(S)}_{H_v^1} + C_S \norm{\Theta_\mu(S)-\Theta_0(S)}_{W_v^{1,\infty}} \norm{a_0(S)}_{H_v^1} \\
&\to 0.
\end{aligned}
\]
This proves \eqref{eq:fixed-time-bmu-continuity-H1}.

Finally, taking the limsup as $\mu\to0$ in the initial decomposition gives
\[
\limsup_{\mu\to0} \norm{Z_\mu-Z_0}_{H_v^1} \leq C\varepsilon S^{-1+C\varepsilon^2}.
\]
Letting $S\to\infty$, and choosing $\varepsilon_0>0$ small enough so that $C\varepsilon^2<1$, we conclude that
\[
\norm{Z_\mu-Z_0}_{H_v^1} \longrightarrow0 \qquad \text{as }\mu\to0.
\]
This proves the lemma.
\end{proof}

\subsection{Logarithmic continuity}
\label{subsec:logarithmic-continuity}

\begin{proposition}
\label{prop:quantitative-continuity-Wmu}
The frozen-profile normalized profiles satisfy
\begin{equation}\label{eq:log-continuity-Wmu}
\log(1/\mu) \norm{W_\mu-W_0}_{H_v^1} \longrightarrow0 \qquad \text{as }\mu\to0.
\end{equation}
\end{proposition}
This is the key continuity estimate needed to replace $W_\mu$ by $W_0$ inside the logarithmic Coulomb phase.

We first prove three lemmas.

\begin{lemma}
\label{lem:quantitative-finite-time-bmu}
There exist constants $\theta>0$, $A>0$, and $C>0$, independent of $0<\mu\leq1$, such that for every $S\geq1$,
\begin{equation}\label{eq:quantitative-finite-time-bmu}
\sup_{1\leq s\leq S} \norm{b_\mu(s)-b_0(s)}_{H_v^1} \leq C\mu^\theta S^A.
\end{equation}
Moreover,
\begin{equation}\label{eq:quantitative-finite-time-bmu-Linfty}
\sup_{1\leq s\leq S} \norm{b_\mu(s)-b_0(s)}_{L_v^\infty} \leq C\mu^\theta S^A.
\end{equation}
\end{lemma}

\begin{proof}
By \Cref{lem:short-time-continuity-mu,lem:polynomial-finite-time-continuity-mu}, for every $S\geq1$,
\begin{equation}\label{eq:finite-time-u-L2-polynomial}
\sup_{0\leq t\leq S} \norm{u_\mu(t)-u_0(t)}_{L_x^2} \leq C\varepsilon^3\mu S^{1+C\varepsilon^2} \leq C\mu S^A.
\end{equation}
Here we may initially take $A=1+C\varepsilon_0^2$ and increase it in the estimates below.

The $\mu$-independent amplitude transformation gives
\[
\norm{a_\mu(t)-a_0(t)}_{L_v^2}=\norm{u_\mu(t)-u_0(t)}_{L_x^2}.
\]
By \Cref{cor:weighted-profile-amplitude-consequences},
\[
\sup_{1\leq t\leq S}\norm{a_\mu(t)-a_0(t)}_{H_v^4}\leq C\varepsilon S^{C\varepsilon^2}.
\]
For $1\leq t\leq S$, interpolation in the inhomogeneous Sobolev spaces yields
\[
\begin{aligned}
\norm{a_\mu(t)-a_0(t)}_{H_v^1} &\leq C\norm{a_\mu(t)-a_0(t)}_{L_v^2}^{3/4}\norm{a_\mu(t)-a_0(t)}_{H_v^4}^{1/4} \leq C\mu^{3/4}S^A,\\
\norm{a_\mu(t)-a_0(t)}_{L_v^\infty} &\leq C\norm{a_\mu(t)-a_0(t)}_{L_v^2}^{1/2}\norm{a_\mu(t)-a_0(t)}_{H_v^4}^{1/2} \leq C\mu^{1/2}S^A.
\end{aligned}
\]
The second line uses $H_v^2\hookrightarrow L_v^\infty$. Taking $\theta=1/2$ at this stage, we obtain
\begin{equation}\label{eq:finite-time-a-H1-continuity}
\sup_{1\leq t\leq S}\norm{a_\mu(t)-a_0(t)}_{H_v^1}\leq C\mu^\theta S^A
\end{equation}
and
\begin{equation}\label{eq:finite-time-a-Linfty-continuity}
\sup_{1\leq t\leq S}\norm{a_\mu(t)-a_0(t)}_{L_v^\infty}\leq C\mu^\theta S^A.
\end{equation}
We may decrease $\theta$ and increase $A$ in the estimates below.

We next compare the exact phases. Recall that
\[
\Theta_\mu(t) = \kappa \int_1^t \frac1s V_{\mu s}*\abs{a_\mu(s)}^2 \,\dd s.
\]
We write
\[
\Theta_\mu(t)-\Theta_0(t) = \kappa \int_1^t \frac1s V_{\mu s}* \left( \abs{a_\mu(s)}^2-\abs{a_0(s)}^2 \right) \,\dd s + \kappa \int_1^t \frac1s (V_{\mu s}-V_0)*\abs{a_0(s)}^2 \,\dd s.
\]
Using \eqref{eq:finite-time-a-H1-continuity}, \eqref{eq:finite-time-a-Linfty-continuity}, and the convolution estimates
\[
\norm{V_\lambda*h}_{L^\infty} \leq C\norm{h}_{L^1}^{2/3}\norm{h}_{L^\infty}^{1/3}, \qquad \norm{\nabla V_\lambda*h}_{L^\infty} \leq C\norm{h}_{L^1}^{1/3}\norm{h}_{L^\infty}^{2/3},
\]
we get
\[
\sup_{1\leq t\leq S} \norm{ \int_1^t \frac1s V_{\mu s}* \left( \abs{a_\mu(s)}^2-\abs{a_0(s)}^2 \right) \,\dd s }_{W_v^{1,\infty}} \leq C\mu^\theta S^A.
\]

For the kernel difference, one uses
\[
0\leq V_0(z)-V_{\mu s}(z) = \frac{1-\ee^{-\mu s\abs{z}}}{\abs{z}} \leq \mu s
\]
and the analogous near-field/far-field estimate for the gradient. This gives
\[
\sup_{1\leq t\leq S} \norm{ \int_1^t \frac1s (V_{\mu s}-V_0)*\abs{a_0(s)}^2 \,\dd s }_{W_v^{1,\infty}} \leq C\mu^\theta S^A.
\]
Consequently,
\begin{equation}\label{eq:finite-time-Theta-W1infty}
\sup_{1\leq t\leq S} \norm{\Theta_\mu(t)-\Theta_0(t)}_{W_v^{1,\infty}} \leq C\mu^\theta S^A.
\end{equation}

Finally,
\[
b_\mu(t)=\ee^{\ii\Theta_\mu(t)}a_\mu(t).
\]
Using \eqref{eq:finite-time-a-H1-continuity}, \eqref{eq:finite-time-a-Linfty-continuity}, and \eqref{eq:finite-time-Theta-W1infty}, we obtain
\[
\sup_{1\leq t\leq S} \norm{b_\mu(t)-b_0(t)}_{H_v^1} \leq C\mu^\theta S^A.
\]
The $L_v^\infty$-estimate follows in the same way. This proves the lemma.
\end{proof}

\begin{lemma}
\label{lem:log-continuity-Zmu}
One has
\begin{equation}\label{eq:log-continuity-Zmu}
\log(1/\mu) \left( \norm{Z_\mu-Z_0}_{H_v^1} + \norm{Z_\mu-Z_0}_{L_v^\infty} \right) \longrightarrow0 \qquad \text{as }\mu\to0.
\end{equation}
\end{lemma}

\begin{proof}
Let $S\geq1$. We decompose
\[
\begin{aligned}
\norm{Z_\mu-Z_0}_{H_v^1} + \norm{Z_\mu-Z_0}_{L_v^\infty} &\leq \norm{Z_\mu-b_\mu(S)}_{H_v^1} + \norm{Z_\mu-b_\mu(S)}_{L_v^\infty} \\
&\quad + \norm{b_\mu(S)-b_0(S)}_{H_v^1} + \norm{b_\mu(S)-b_0(S)}_{L_v^\infty} \\
&\quad + \norm{b_0(S)-Z_0}_{H_v^1} + \norm{b_0(S)-Z_0}_{L_v^\infty}.
\end{aligned}
\]
By the uniform tail estimates \eqref{eq:Zmu-tail-H1-bound} and \eqref{eq:exact-profile-tail-bound},
\[
\norm{Z_\mu-b_\mu(S)}_{H_v^1} + \norm{Z_\mu-b_\mu(S)}_{L_v^\infty} + \norm{b_0(S)-Z_0}_{H_v^1} + \norm{b_0(S)-Z_0}_{L_v^\infty} \leq C\varepsilon S^{-1+C\varepsilon^2}.
\]
By \Cref{lem:quantitative-finite-time-bmu},
\[
\norm{b_\mu(S)-b_0(S)}_{H_v^1} + \norm{b_\mu(S)-b_0(S)}_{L_v^\infty} \leq C\mu^\theta S^A.
\]
Hence
\begin{equation}\label{eq:Zmu-Z0-H1-Linfty-quantitative-cutoff}
\norm{Z_\mu-Z_0}_{H_v^1} + \norm{Z_\mu-Z_0}_{L_v^\infty} \leq C\varepsilon S^{-1+C\varepsilon^2} + C\mu^\theta S^A.
\end{equation}
For $0<\mu<\ee^{-1}$, take $S=(\log(1/\mu))^p$, where $p(1-C\varepsilon_0^2)>1$. Multiplying \eqref{eq:Zmu-Z0-H1-Linfty-quantitative-cutoff} by $\log(1/\mu)$ gives
\[
\begin{aligned}
&\log(1/\mu)\bigl(\norm{Z_\mu-Z_0}_{H_v^1}+\norm{Z_\mu-Z_0}_{L_v^\infty}\bigr)\\
&\qquad\leq C\varepsilon(\log(1/\mu))^{1-p(1-C\varepsilon^2)} +C\mu^\theta(\log(1/\mu))^{1+pA}\longrightarrow0.
\end{aligned}
\]
The first term tends to zero by the choice of $p$, and the second tends to zero because $\mu^\theta$ dominates every fixed power of $\log(1/\mu)$.
\end{proof}

\begin{lemma}
\label{lem:log-Gamma-Y-convergence}
One has
\begin{equation}\label{eq:log-Gamma-Y-convergence}
\log(1/\mu) \norm{\Gamma_\mu^Y-\Gamma_0^Y}_{W_v^{1,\infty}} \longrightarrow0 \qquad \text{as }\mu\to0.
\end{equation}
\end{lemma}

\begin{proof}
For $S\geq1$, use the tail $R_\mu(S)$ defined in \eqref{eq:Rmu-tail-definition}. Write $h_\mu(s)=\abs{a_\mu(s)}^2-\abs{Z_\mu}^2$. The definitions give
\[
\begin{aligned}
\Gamma_\mu^Y-\Gamma_0^Y={}&R_\mu(S)-R_0(S)\\
&+\kappa\int_1^S V_{\mu s}*(h_\mu(s)-h_0(s))\frac{\dd s}{s} +\kappa\int_1^S(V_{\mu s}-V_0)*h_0(s)\frac{\dd s}{s}.
\end{aligned}
\]
The density tail in the proof of \Cref{prop:frozen-profile-asymptotics} gives $\norm{h_\mu(s)}_{L^1}+\norm{h_\mu(s)}_{L^\infty}\leq C\varepsilon^2s^{-1+C\varepsilon^2}$. Splitting into $\abs{z}\leq1$ and $\abs{z}>1$, the bounds $V_\lambda(z)\leq\abs{z}^{-1}$ and $\abs{\nabla V_\lambda(z)}\leq C\abs{z}^{-2}$ imply
\[
\norm{V_\lambda*h}_{W_v^{1,\infty}}\leq C\bigl(\norm{h}_{L^1}+\norm{h}_{L^\infty}\bigr),\qquad\lambda\geq0.
\]
Integrating this estimate for $h=h_\mu(s)$ yields
\begin{equation}\label{eq:Gamma-tail-W1infty}
\sup_{0\leq\mu\leq1}\norm{R_\mu(S)}_{W_v^{1,\infty}}\leq C\varepsilon^2S^{-1+C\varepsilon^2}.
\end{equation}
Hence $\norm{R_\mu(S)-R_0(S)}_{W_v^{1,\infty}}\leq C\varepsilon^2S^{-1+C\varepsilon^2}$. It remains to estimate the two integrals over $[1,S]$.

We first estimate $h_\mu-h_0$. Since $\abs{a_\mu}=\abs{b_\mu}$, we have
\[
h_\mu(s) = \abs{b_\mu(s)}^2-\abs{Z_\mu}^2.
\]
Therefore
\[
\begin{aligned}
h_\mu(s)-h_0(s) &= \left( \abs{b_\mu(s)}^2-\abs{b_0(s)}^2 \right) - \left( \abs{Z_\mu}^2-\abs{Z_0}^2 \right).
\end{aligned}
\]
Using the uniform $L^2\cap L^\infty$ bounds on $b_\mu(s)$ and $Z_\mu$, together with \Cref{lem:quantitative-finite-time-bmu}, we obtain, for $1\leq s\leq S$,
\[
\norm{h_\mu(s)-h_0(s)}_{L^1} \leq C\mu^\theta S^A + C\norm{Z_\mu-Z_0}_{L^2_v},
\]
and
\[
\norm{h_\mu(s)-h_0(s)}_{L^\infty} \leq C\mu^\theta S^A + C\norm{Z_\mu-Z_0}_{L^\infty_v}.
\]
Using the convolution bound above, we obtain
\[
\norm{ \int_1^S \frac1s V_{\mu s}*(h_\mu(s)-h_0(s)) \,\dd s }_{W_v^{1,\infty}} \leq C(\log S) \left( \mu^\theta S^A + \norm{Z_\mu-Z_0}_{H_v^1} + \norm{Z_\mu-Z_0}_{L_v^\infty} \right).
\]

For the kernel-difference term, we use the quantitative kernel convergence
\[
\norm{(V_{\mu s}-V_0)*h_0(s)}_{W_v^{1,\infty}} \leq C\mu^\theta s^A \norm{h_0(s)}_{L^1\cap L^\infty},
\]
which follows from the pointwise bound
\[
0\leq V_0(z)-V_{\mu s}(z) \leq \mu s
\]
and the corresponding near-field/far-field estimate for the gradient. Since
\[
\norm{h_0(s)}_{L^1\cap L^\infty} \leq C\varepsilon^2s^{-1+C\varepsilon^2},
\]
we obtain
\[
\norm{ \int_1^S \frac1s (V_{\mu s}-V_0)*h_0(s) \,\dd s }_{W_v^{1,\infty}} \leq C\mu^\theta S^A.
\]

Combining the finite and tail estimates yields
\begin{equation}\label{eq:Gamma-difference-cutoff-bound}
\norm{\Gamma_\mu^Y-\Gamma_0^Y}_{W_v^{1,\infty}} \leq C\varepsilon^2S^{-1+C\varepsilon^2} + C(\log S) \left( \mu^\theta S^A + \norm{Z_\mu-Z_0}_{H_v^1} + \norm{Z_\mu-Z_0}_{L_v^\infty} \right) + C\mu^\theta S^A.
\end{equation}

For $0<\mu<\ee^{-1}$, take $S=(\log(1/\mu))^p$, where $p(1-C\varepsilon_0^2)>1$. Multiplying \eqref{eq:Gamma-difference-cutoff-bound} by $\log(1/\mu)$ gives
\[
\begin{aligned}
\log(1/\mu)\norm{\Gamma_\mu^Y-\Gamma_0^Y}_{W_v^{1,\infty}} \leq{}&C\varepsilon^2\log(1/\mu)S^{-1+C\varepsilon^2}\\
&+C\log(1/\mu)(1+\log S)\mu^\theta S^A\\
&+C\log(1/\mu)\log S\bigl(\norm{Z_\mu-Z_0}_{H_v^1}+\norm{Z_\mu-Z_0}_{L_v^\infty}\bigr).
\end{aligned}
\]
The first two terms tend to zero. For the last term, use the quantitative estimate \eqref{eq:Zmu-Z0-H1-Linfty-quantitative-cutoff} with this same $S$. Since $\log S=p\log\log(1/\mu)$, we obtain
\[
\begin{aligned}
&\log(1/\mu)\log S\bigl(\norm{Z_\mu-Z_0}_{H_v^1}+\norm{Z_\mu-Z_0}_{L_v^\infty}\bigr)\\
&\qquad\leq C\varepsilon(\log(1/\mu))^{1-p(1-C\varepsilon^2)}\log\log(1/\mu)
+C\mu^\theta(\log(1/\mu))^{1+pA}\log\log(1/\mu)\longrightarrow0.
\end{aligned}
\]
This proves \eqref{eq:log-Gamma-Y-convergence}.
\end{proof}

We can now prove \Cref{prop:quantitative-continuity-Wmu}.

\begin{proof}[Proof of \Cref{prop:quantitative-continuity-Wmu}]
Recall that
\[
W_\mu = \ee^{-\ii\Gamma_\mu^Y}Z_\mu, \qquad W_0 = \ee^{-\ii\Gamma_0^Y}Z_0.
\]
Hence
\[
W_\mu-W_0 = \ee^{-\ii\Gamma_\mu^Y}(Z_\mu-Z_0) + \left( \ee^{-\ii\Gamma_\mu^Y} - \ee^{-\ii\Gamma_0^Y} \right)Z_0.
\]
Since $\Gamma_\mu^Y$ is real-valued and uniformly bounded in $W_v^{1,\infty}$, multiplication by $\ee^{-\ii\Gamma_\mu^Y}$ is uniformly bounded on $H_v^1$. Therefore
\[
\norm{ \ee^{-\ii\Gamma_\mu^Y}(Z_\mu-Z_0) }_{H_v^1} \leq C\norm{Z_\mu-Z_0}_{H_v^1}.
\]
Also,
\[
\norm{ \ee^{-\ii\Gamma_\mu^Y} - \ee^{-\ii\Gamma_0^Y} }_{W_v^{1,\infty}} \leq C \norm{\Gamma_\mu^Y-\Gamma_0^Y}_{W_v^{1,\infty}},
\]
and hence
\[
\norm{ \left( \ee^{-\ii\Gamma_\mu^Y} - \ee^{-\ii\Gamma_0^Y} \right)Z_0 }_{H_v^1} \leq C \norm{\Gamma_\mu^Y-\Gamma_0^Y}_{W_v^{1,\infty}} \norm{Z_0}_{H_v^1}.
\]
Combining these estimates gives
\[
\norm{W_\mu-W_0}_{H_v^1} \leq C \norm{Z_\mu-Z_0}_{H_v^1} + C \norm{\Gamma_\mu^Y-\Gamma_0^Y}_{W_v^{1,\infty}}.
\]
Multiplying by $\log(1/\mu)$ and using \Cref{lem:log-continuity-Zmu,lem:log-Gamma-Y-convergence}, we obtain
\[
\log(1/\mu) \norm{W_\mu-W_0}_{H_v^1} \longrightarrow0.
\]
This proves the proposition.
\end{proof}

\section{Proof of the trichotomy}
\label{sec:proof-trichotomy}

We establish the required estimates for the phase functionals and then prove the trichotomy.

\subsection{Phase functional estimates}
\label{sec:phase-functional-estimates}

\begin{definition}[Profile topology]
We say that $G_n\to G$ in the profile topology if
\[
G_n\to G \quad\text{strongly in }H_v^1, \qquad \sup_n \norm{G_n}_{L_v^\infty} + \norm{G}_{L_v^\infty} <\infty.
\]
Equivalently, the convergence is strong in $H_v^1$ along a uniformly $L_v^\infty$-bounded profile class.
\end{definition}

\begin{lemma}
\label{lem:C-functional-continuity}
If $G_n\to G$ in the profile topology, then
\begin{equation}\label{eq:C-functional-continuity}
\norm{\cC[G_n]-\cC[G]}_{L_v^\infty} \longrightarrow0.
\end{equation}
\end{lemma}
This lemma controls the logarithmic Coulomb phase under convergence of profiles.

\begin{proof}
We write
\[
\cC[G_n](v)-\cC[G](v) = \kappa \int_{\R^3} \frac{\abs{G_n(w)}^2-\abs{G(w)}^2}{\abs{v-w}} \,\dd w.
\]
Set
\[
h_n := \abs{G_n}^2-\abs{G}^2.
\]
It is enough to prove
\[
\snorm{\abs{\,\cdot\,}^{-1}*h_n}_{L_v^\infty} \longrightarrow0.
\]

First,
\[
\norm{h_n}_{L^1} \leq \left( \norm{G_n}_{L^2} + \norm{G}_{L^2} \right) \norm{G_n-G}_{L^2}.
\]
Since $G_n\to G$ in $H_v^1$, in particular $G_n\to G$ in $L_v^2$. Moreover the $L^2$-norms are bounded. Hence
\begin{equation}\label{eq:C-continuity-hn-L1}
\norm{h_n}_{L^1} \longrightarrow0.
\end{equation}
On the other hand, the profile class under consideration is uniformly bounded in $L^\infty$, and therefore
\begin{equation}\label{eq:C-continuity-hn-Linfty}
\norm{h_n}_{L^\infty} \leq \norm{G_n}_{L^\infty}^2 + \norm{G}_{L^\infty}^2 \leq C.
\end{equation}
We use the elementary Coulomb convolution estimate
\begin{equation}\label{eq:Coulomb-convolution-L1-Linfty}
\snorm{\abs{\,\cdot\,}^{-1}*h}_{L^\infty} \leq C \norm{h}_{L^1}^{2/3} \norm{h}_{L^\infty}^{1/3}.
\end{equation}
Applying \eqref{eq:Coulomb-convolution-L1-Linfty} to $h=h_n$, and using \eqref{eq:C-continuity-hn-L1} and \eqref{eq:C-continuity-hn-Linfty}, we obtain
\[
\snorm{ \abs{\,\cdot\,}^{-1}*h_n }_{L^\infty} \leq C \norm{h_n}_{L^1}^{2/3} \norm{h_n}_{L^\infty}^{1/3} \longrightarrow0.
\]
Therefore
\[
\norm{\cC[G_n]-\cC[G]}_{L_v^\infty} \longrightarrow0.
\]
This proves the lemma.
\end{proof}

\begin{lemma}
\label{lem:BL-functional-continuity}
For each fixed $L\in[0,\infty)$, if $G_n\to G$ in the profile topology, then
\begin{equation}\label{eq:BL-functional-continuity}
\norm{\cB_L[G_n]-\cB_L[G]}_{L_v^\infty} \longrightarrow0.
\end{equation}
Moreover, this convergence is uniform for $L$ in compact subsets of $[0,\infty)$.
\end{lemma}
This lemma controls the finite transition regime $\mu t\to L$.

\begin{proof}
We write
\[
\cB_L[G_n](v)-\cB_L[G](v) = \kappa \int_{\R^3} b_L(v-w) \left( \abs{G_n(w)}^2-\abs{G(w)}^2 \right) \,\dd w.
\]
Set
\[
h_n := \abs{G_n}^2-\abs{G}^2.
\]
For $L\geq0$, since
\[
b_L(z) = \frac1{\abs{z}} \int_0^L \frac{\ee^{-r\abs{z}}-1}{r} \,\dd r,
\]
and
\[
|\ee^{-r\abs{z}}-1| \leq r\abs{z},
\]
we get
\begin{equation}\label{eq:bL-basic-uniform-bound}
\abs{b_L(z)} \leq \frac1{\abs{z}} \int_0^L \abs{z} \,\dd r = L.
\end{equation}
Therefore
\[
\norm{\cB_L[G_n]-\cB_L[G]}_{L_v^\infty} \leq L\norm{h_n}_{L^1}.
\]
Moreover,
\[
\norm{h_n}_{L^1} \leq \left( \norm{G_n}_{L^2} + \norm{G}_{L^2} \right) \norm{G_n-G}_{L^2}.
\]
Since $G_n\to G$ in the profile topology, in particular $G_n\to G$ in $L_v^2$, and the $L^2$-norms are uniformly bounded. Hence
\[
\norm{h_n}_{L^1} \longrightarrow0.
\]
Thus, for every fixed $L\in[0,\infty)$,
\[
\norm{\cB_L[G_n]-\cB_L[G]}_{L_v^\infty} \longrightarrow0.
\]

It remains to prove the uniformity for $L$ in compact subsets of $[0,\infty)$. Let $L_*>0$. If $0\leq L\leq L_*$, then \eqref{eq:bL-basic-uniform-bound} gives
\[
\abs{b_L(z)} \leq L_* \qquad \text{for all }z\neq0.
\]
Therefore
\[
\sup_{0\leq L\leq L_*} \norm{\cB_L[G_n]-\cB_L[G]}_{L_v^\infty} \leq L_* \norm{h_n}_{L^1} \longrightarrow0.
\]
This proves the uniform convergence for $L$ in compact subsets of $[0,\infty)$.
\end{proof}

\begin{lemma}
\label{lem:Binfty-functional-continuity}
If $G_n\to G$ in the profile topology, then
\begin{equation}\label{eq:Binfty-functional-continuity}
\norm{\cB_\infty[G_n]-\cB_\infty[G]}_{L_v^\infty} \longrightarrow0.
\end{equation}
\end{lemma}
This lemma controls the saturated phase in the regime $\mu t\to\infty$.

\begin{proof}
We write
\[
\cB_\infty[G_n](v)-\cB_\infty[G](v) = \kappa \int_{\R^3} b_\infty(v-w) \left( \abs{G_n(w)}^2-\abs{G(w)}^2 \right) \,\dd w.
\]
Set
\[
h_n := \abs{G_n}^2-\abs{G}^2.
\]
It is enough to prove
\[
\norm{b_\infty*h_n}_{L_v^\infty} \longrightarrow0.
\]

First,
\[
\norm{h_n}_{L^1} \leq \left( \norm{G_n}_{L^2} + \norm{G}_{L^2} \right) \norm{G_n-G}_{L^2}.
\]
Since $G_n\to G$ in the profile topology, in particular $G_n\to G$ in $L_v^2$, and the $L^2$-norms are uniformly bounded. Hence
\begin{equation}\label{eq:Binfty-continuity-hn-L1}
\norm{h_n}_{L^1} \longrightarrow0.
\end{equation}
Moreover, the profile topology gives a uniform $L^\infty$-bound, and hence
\begin{equation}\label{eq:Binfty-continuity-hn-Linfty}
\norm{h_n}_{L^\infty} \leq \norm{G_n}_{L^\infty}^2 + \norm{G}_{L^\infty}^2 \leq C.
\end{equation}
Also,
\[
\norm{h_n}_{L^1} \leq C
\]
uniformly in $n$.

Fix $0<\rho<1/2$ and $R>2$. We decompose
\[
b_\infty*h_n = (\mathbf 1_{\{\abs{\,\cdot\,}\leq\rho\}}b_\infty)*h_n + (\mathbf 1_{\{\rho<\abs{\,\cdot\,}<R\}}b_\infty)*h_n + (\mathbf 1_{\{\abs{\,\cdot\,}\geq R\}}b_\infty)*h_n.
\]

For the near region, using \eqref{eq:Binfty-continuity-hn-Linfty}, we get
\[
\begin{aligned}
\sup_v \int_{\abs{v-w}\leq\rho} \abs{b_\infty(v-w)} \abs{h_n(w)} \,\dd w &\leq \norm{h_n}_{L^\infty} \int_{\abs{z}\leq\rho} \frac{\abs{\gamma+\log\abs{z}}}{\abs{z}} \,\dd z \\
&\leq C \int_0^\rho r\abs{\log r} \,\dd r \\
&\leq C\rho^2\abs{\log\rho}.
\end{aligned}
\]
This bound is uniform in $n$.

For the far region, if $R$ is sufficiently large, then
\[
\sup_{\abs{z}\geq R} \abs{b_\infty(z)} \leq C\frac{\log R}{R}.
\]
Thus
\[
\sup_v \int_{\abs{v-w}\geq R} \abs{b_\infty(v-w)} \abs{h_n(w)} \,\dd w \leq \sup_{\abs{z}\geq R}\abs{b_\infty(z)} \norm{h_n}_{L^1} \leq C\frac{\log R}{R}.
\]
Again this is uniform in $n$.

It remains to estimate the middle region. On the annulus
\[
\rho\leq\abs{z}\leq R,
\]
the kernel $b_\infty$ is bounded. Hence
\[
\begin{aligned}
\sup_v \int_{\rho<\abs{v-w}<R} \abs{b_\infty(v-w)} \abs{h_n(w)} \,\dd w &\leq C_{\rho,R}\norm{h_n}_{L^1}.
\end{aligned}
\]
By \eqref{eq:Binfty-continuity-hn-L1}, this tends to zero as $n\to\infty$, for fixed $\rho$ and $R$.

Combining the three estimates, we obtain
\[
\limsup_{n\to\infty} \norm{b_\infty*h_n}_{L_v^\infty} \leq C\rho^2\abs{\log\rho} + C\frac{\log R}{R}.
\]
Now let first $\rho\to0$, and then $R\to\infty$. This gives
\[
\norm{b_\infty*h_n}_{L_v^\infty} \longrightarrow0.
\]
Therefore
\[
\norm{\cB_\infty[G_n]-\cB_\infty[G]}_{L_v^\infty} \longrightarrow0.
\]
This proves the lemma.
\end{proof}

\begin{lemma}
\label{lem:uniform-large-L-limit}
Let $\mathcal A\subset L^2(\R^3)\cap L^\infty(\R^3)$ satisfy
\begin{equation}\label{eq:uniform-profile-Linfty-class}
\sup_{G\in\mathcal A}\norm{G}_{L_v^\infty}<\infty.
\end{equation}
Then
\begin{equation}\label{eq:uniform-large-L-limit}
\sup_{G\in\mathcal A} \norm{\cB_L[G]+(\log L)\cC[G]-\cB_\infty[G]}_{L_v^\infty} \longrightarrow0 \qquad\text{as }L\to\infty.
\end{equation}
More precisely, for every $L>0$,
\begin{equation}\label{eq:uniform-large-L-rate}
\sup_{G\in\mathcal A} \norm{\cB_L[G]+(\log L)\cC[G]-\cB_\infty[G]}_{L_v^\infty} \leq \frac{2\pi}{L^2} \sup_{G\in\mathcal A}\norm{G}_{L_v^\infty}^2.
\end{equation}
\end{lemma}
This lemma is needed to pass from $\cB_{\mu t}$ to $\cB_\infty$ in the large regime.

\begin{proof}
For each $G\in L^2\cap L^\infty$, the three phase functionals are well defined. For $L>0$ and $z\neq0$, the elementary identity
\[
\int_0^x\frac{\ee^{-s}-1}{s}\dd s+\log x+\gamma =-\int_x^\infty\frac{\ee^{-s}}s\dd s,\qquad x>0,
\]
gives
\begin{equation}\label{eq:transition-kernel-tail}
b_L(z)+\frac{\log L}{\abs{z}}-b_\infty(z) =-\int_L^\infty\frac{V_s(z)}s\dd s\leq0.
\end{equation}
Tonelli's theorem therefore yields
\[
\cB_L[G]+(\log L)\cC[G]-\cB_\infty[G] =-\kappa\int_L^\infty(V_s*\abs{G}^2)\frac{\dd s}{s}.
\]
Since $\norm{V_s}_{L^1}=4\pi s^{-2}$, Young's inequality gives
\[
\begin{aligned}
\norm{\cB_L[G]+(\log L)\cC[G]-\cB_\infty[G]}_{L_v^\infty} &\leq\norm{G}_{L_v^\infty}^2\int_L^\infty\norm{V_s}_{L^1}\frac{\dd s}{s}\\
&=\frac{2\pi}{L^2}\norm{G}_{L_v^\infty}^2.
\end{aligned}
\]
Taking the supremum over $G\in\mathcal A$ proves the claim. Individual $L^2$ membership defines the separate functionals, while the displayed bound uses only the uniform $L^\infty$ bound.
\end{proof}

\subsection{The three asymptotic regimes}
\label{sec:proof-main-theorem}

We now apply the frozen-profile asymptotic formula and examine the three possible limits of $\mu_nt_n$.

The following proposition reduces the proof of the theorem to phase convergence.
\begin{proposition}
\label{prop:master-comparison}
Let $\mu_n\to0$ and $t_n\to\infty$. Then
\begin{equation}\label{eq:master-comparison}
\norm{ a_{\mu_n}(t_n) - \ee^{-\ii\Phi_{\mu_n}[W_{\mu_n}](t_n)}W_{\mu_n} }_{L_v^2} \longrightarrow0.
\end{equation}
Moreover, the following regime-dependent replacements hold.

If $\mu_nt_n\to0$, then
\begin{equation}\label{eq:master-comparison-small-mut}
\norm{ \ee^{-\ii\Phi_{\mu_n}[W_{\mu_n}](t_n)}W_{\mu_n} - \ee^{-\ii\cC[W_0](\log t_n)}W_0 }_{L_v^2} \longrightarrow0.
\end{equation}

If $\mu_nt_n\to L\in(0,\infty)$, then
\begin{equation}\label{eq:master-comparison-finite-mut}
\norm{ \ee^{-\ii\Phi_{\mu_n}[W_{\mu_n}](t_n)}W_{\mu_n} - \ee^{-\ii\cC[W_0](\log t_n)} \ee^{-\ii\cB_L[W_0]} W_0 }_{L_v^2} \longrightarrow0.
\end{equation}

If $\mu_nt_n\to\infty$, then
\begin{equation}\label{eq:master-comparison-large-mut}
\norm{ \ee^{-\ii\Phi_{\mu_n}[W_{\mu_n}](t_n)}W_{\mu_n} - \ee^{-\ii\cC[W_0]\log(1/\mu_n)} \ee^{-\ii\cB_\infty[W_0]} W_0 }_{L_v^2} \longrightarrow0.
\end{equation}
\end{proposition}

\begin{proof}
The first assertion \eqref{eq:master-comparison} follows directly from \Cref{prop:frozen-profile-asymptotics}, i.e.,
\[
\norm{ a_{\mu_n}(t_n) - \ee^{-\ii\Phi_{\mu_n}[W_{\mu_n}](t_n)}W_{\mu_n} }_{L_v^2} \leq C\varepsilon t_n^{-1+C\varepsilon^2} \to0.
\]

For the Coulomb phase, we recall that, for two profiles $F,G\in H_v^1$, Hardy's inequality gives, uniformly in $v$,
\[
\int_{\R^3} \frac{\abs{F(w)-G(w)}\abs{F(w)+G(w)}}{\abs{v-w}} \,\dd w \leq C \norm{F-G}_{H_v^1} \norm{F+G}_{L_v^2}.
\]
Therefore
\begin{equation}\label{eq:Coulomb-phase-H1-Lipschitz}
\norm{\cC[F]-\cC[G]}_{L_v^\infty} \leq C \left( \norm{F}_{L_v^2} + \norm{G}_{L_v^2} \right) \norm{F-G}_{H_v^1}.
\end{equation}
Since the profiles $W_\mu$ are uniformly bounded in $L_v^2$, \Cref{prop:quantitative-continuity-Wmu} implies
\begin{equation}\label{eq:log-CWmu-CW}
\log(1/\mu) \norm{\cC[W_\mu]-\cC[W_0]}_{L_v^\infty} \longrightarrow0.
\end{equation}

We also use the elementary bound, using $\abs{b_\ell(z)}\leq \ell$,
\begin{equation}\label{eq:BL-small-bound-master}
\norm{\cB_\ell[F]}_{L_v^\infty} \leq C\ell\norm{F}_{L_v^2}^2, \qquad \ell\geq0.
\end{equation}

We now consider the three regimes.

We first consider the regime $\mu_nt_n\to0$.
For indices with $\mu_n=0$, the expression in \eqref{eq:master-comparison-small-mut} is zero because $\Phi_0[W_0](t_n)=\cC[W_0]\log t_n$. It therefore suffices to consider indices with $\mu_n>0$.

Using the frozen phase decomposition,
\[
\Phi_{\mu_n}[W_{\mu_n}](t_n) = \cC[W_{\mu_n}]\log t_n + \cB_{\mu_nt_n}[W_{\mu_n}] - \cB_{\mu_n}[W_{\mu_n}].
\]
Since $\mu_nt_n\to0$, we have $\log t_n\leq \log(1/\mu_n)$ for all large $n$. Thus, by \eqref{eq:log-CWmu-CW},
\[
\log t_n \norm{\cC[W_{\mu_n}]-\cC[W_0]}_{L_v^\infty} \to0.
\]
Moreover, by \eqref{eq:BL-small-bound-master},
\[
\norm{\cB_{\mu_nt_n}[W_{\mu_n}]}_{L_v^\infty} + \norm{\cB_{\mu_n}[W_{\mu_n}]}_{L_v^\infty} \leq C\varepsilon^2(\mu_nt_n+\mu_n) \to0.
\]
Hence
\[
\norm{ \Phi_{\mu_n}[W_{\mu_n}](t_n) - \cC[W_0]\log t_n }_{L_v^\infty} \to0.
\]
Since $W_{\mu_n}\to W_0$ in $H_v^1$, hence in $L_v^2$, we get
\[
\begin{aligned}
&\norm{ \ee^{-\ii\Phi_{\mu_n}[W_{\mu_n}](t_n)}W_{\mu_n} - \ee^{-\ii\cC[W_0](\log t_n)}W_0 }_{L_v^2} \\
&\leq \norm{W_{\mu_n}-W_0}_{L_v^2} + \norm{ \left( \ee^{-\ii\Phi_{\mu_n}[W_{\mu_n}](t_n)} - \ee^{-\ii\cC[W_0](\log t_n)} \right)W_0 }_{L_v^2} \\
&\leq \norm{W_{\mu_n}-W_0}_{L_v^2} + \norm{ \Phi_{\mu_n}[W_{\mu_n}](t_n) - \cC[W_0]\log t_n }_{L_v^\infty} \norm{W_0}_{L_v^2} \to0.
\end{aligned}
\]
This proves \eqref{eq:master-comparison-small-mut}.

Next assume $\mu_nt_n\to L\in(0,\infty)$. Then
\[
\log t_n = \log(1/\mu_n) + \log(\mu_nt_n),
\]
and $\log(\mu_nt_n)$ is bounded. Therefore \eqref{eq:log-CWmu-CW} gives
\[
\log t_n \norm{\cC[W_{\mu_n}]-\cC[W_0]}_{L_v^\infty} \to0.
\]
Also,
\[
\norm{\cB_{\mu_n}[W_{\mu_n}]}_{L_v^\infty} \leq C\varepsilon^2\mu_n \to0.
\]
For the finite transition term, by $\mu_nt_n\to L$, the uniform compact-$L$ continuity in \Cref{lem:BL-functional-continuity}, and the elementary continuity in the parameter $L$, we have
\[
\norm{\cB_{\mu_nt_n}[W_{\mu_n}]-\cB_L[W_0]}_{L_v^\infty} \to0.
\]
Consequently,
\[
\norm{ \Phi_{\mu_n}[W_{\mu_n}](t_n) - \cC[W_0]\log t_n - \cB_L[W_0] }_{L_v^\infty} \to0.
\]
As in the previous case, this implies \eqref{eq:master-comparison-finite-mut}.

Finally assume $\mu_nt_n\to\infty$. We use
\[
\log t_n=\log(1/\mu_n)+\log(\mu_nt_n).
\]
Using the frozen phase decomposition,
\[
\begin{aligned}
\Phi_{\mu_n}[W_{\mu_n}](t_n) &= \cC[W_{\mu_n}]\log t_n + \cB_{\mu_nt_n}[W_{\mu_n}] - \cB_{\mu_n}[W_{\mu_n}] \\
&= \cC[W_{\mu_n}]\log(1/\mu_n) + \left( \cB_{\mu_nt_n}[W_{\mu_n}] + (\log(\mu_nt_n))\cC[W_{\mu_n}] \right) - \cB_{\mu_n}[W_{\mu_n}].
\end{aligned}
\]
By \eqref{eq:log-CWmu-CW},
\[
\log(1/\mu_n) \norm{\cC[W_{\mu_n}]-\cC[W_0]}_{L_v^\infty} \to0.
\]
By \Cref{lem:uniform-large-L-limit}, applied to the family $\{W_{\mu_n}\}_n\subset L^2\cap L^\infty$, which is uniformly bounded in $L^\infty$,
\[
\norm{ \cB_{\mu_nt_n}[W_{\mu_n}] + (\log(\mu_nt_n))\cC[W_{\mu_n}] - \cB_\infty[W_{\mu_n}] }_{L_v^\infty} \to0.
\]
By \Cref{lem:Binfty-functional-continuity},
\[
\norm{\cB_\infty[W_{\mu_n}]-\cB_\infty[W_0]}_{L_v^\infty} \to0.
\]
Finally,
\[
\norm{\cB_{\mu_n}[W_{\mu_n}]}_{L_v^\infty} \leq C\varepsilon^2\mu_n \to0.
\]
Combining these estimates gives
\[
\norm{ \Phi_{\mu_n}[W_{\mu_n}](t_n) - \cC[W_0]\log(1/\mu_n) - \cB_\infty[W_0] }_{L_v^\infty} \to0.
\]
Together with $W_{\mu_n}\to W_0$ in $L_v^2$, this proves \eqref{eq:master-comparison-large-mut}. This proves the proposition.
\end{proof}

We can now prove \Cref{thm:main-trichotomy}.

\begin{proof}[Proof of \Cref{thm:main-trichotomy}]
Let $\mu_n\to0$ and $t_n\to\infty$. We first work at the level of the rescaled amplitudes. By \Cref{prop:master-comparison}, we have
\begin{equation}\label{eq:main-proof-master-amplitude}
\norm{ a_{\mu_n}(t_n) - \ee^{-\ii\Phi_{\mu_n}[W_{\mu_n}](t_n)}W_{\mu_n} }_{L_v^2} \longrightarrow0.
\end{equation}

Assume first that $\mu_nt_n\to0$. Then \Cref{prop:master-comparison} gives
\[
\norm{ \ee^{-\ii\Phi_{\mu_n}[W_{\mu_n}](t_n)}W_{\mu_n} - \ee^{-\ii\cC[W_0]\log t_n}W_0 }_{L_v^2} \longrightarrow0.
\]
Combining this with \eqref{eq:main-proof-master-amplitude}, we obtain
\begin{equation}\label{eq:main-proof-amplitude-small-mut}
a_{\mu_n}(t_n,v) = \ee^{-\ii\cC[W_0](v)\log t_n} W_0(v) + o_{L_v^2}(1).
\end{equation}

Assume next that $\mu_nt_n\to L\in(0,\infty)$. Then \Cref{prop:master-comparison} gives
\[
\norm{ \ee^{-\ii\Phi_{\mu_n}[W_{\mu_n}](t_n)}W_{\mu_n} - \ee^{-\ii\cC[W_0]\log t_n} \ee^{-\ii\cB_L[W_0]} W_0 }_{L_v^2} \longrightarrow0.
\]
Together with \eqref{eq:main-proof-master-amplitude}, this yields
\begin{equation}\label{eq:main-proof-amplitude-finite-mut}
a_{\mu_n}(t_n,v) = \ee^{-\ii\cC[W_0](v)\log t_n} \ee^{-\ii\cB_L[W_0](v)} W_0(v) + o_{L_v^2}(1).
\end{equation}

Finally assume that $\mu_nt_n\to\infty$. Then \Cref{prop:master-comparison} gives
\[
\norm{ \ee^{-\ii\Phi_{\mu_n}[W_{\mu_n}](t_n)}W_{\mu_n} - \ee^{-\ii\cC[W_0]\log(1/\mu_n)} \ee^{-\ii\cB_\infty[W_0]} W_0 }_{L_v^2} \longrightarrow0.
\]
Combining this with \eqref{eq:main-proof-master-amplitude}, we get
\begin{equation}\label{eq:main-proof-amplitude-large-mut}
a_{\mu_n}(t_n,v) = \ee^{-\ii\cC[W_0](v)\log(1/\mu_n)} \ee^{-\ii\cB_\infty[W_0](v)} W_0(v) + o_{L_v^2}(1).
\end{equation}

It remains only to return to physical space. By the relation
\[
u_\mu(t,x) = (\ii t)^{-3/2} \ee^{\frac{\ii\abs{x}^2}{2t}} a_\mu(t,x/t),
\]
the map
\[
A(v) \mapsto (\ii t)^{-3/2} \ee^{\frac{\ii\abs{x}^2}{2t}} A(x/t)
\]
is unitary from $L_v^2$ to $L_x^2$. Applying this unitary map to \eqref{eq:main-proof-amplitude-small-mut}, \eqref{eq:main-proof-amplitude-finite-mut}, and \eqref{eq:main-proof-amplitude-large-mut} gives respectively \eqref{eq:main-theorem-regime-small-mut}, \eqref{eq:main-theorem-regime-finite-mut}, and \eqref{eq:main-theorem-regime-large-mut}. This proves the theorem.
\end{proof}

\section*{Acknowledgment}
We thank the anonymous referee for a careful review, pointing out errors, and providing valuable suggestions to simplify the proof. Y. Cho was supported by the National Research Foundation of Korea(NRF) grant funded by the Korea government(MSIT) (RS-2024-00333393). J. Lee was partially supported by the Global-Learning \& Academic research institution for Master's-PhD students, and Postdocs (G-LAMP) Program of the National Research Foundation of Korea (NRF) grant funded by the Ministry of Education (No.\,RS-2025-25442355), the NRF grant funded by the Korea government (MSIT) (No.\,RS-2026-25589357), and a grant from Kyung Hee University in 2026 (KHU-20262263).

\end{document}